\documentclass[]{interact}

\usepackage[T1]{fontenc}
\usepackage{epstopdf}% To incorporate .eps illustrations using PDFLaTeX, etc.
\usepackage[caption=false]{subfig}% Support for small, `sub' figures and tables
\usepackage{graphicx}
\usepackage{float}
\usepackage[numbers,sort&compress]{natbib}% Citation support using natbib.sty
\bibpunct[, ]{[}{]}{,}{n}{,}{,}% Citation support using natbib.sty
\makeatletter% @ becomes a letter
\def\NAT@def@citea{\def\@citea{\NAT@separator}}% Suppress spaces between citations using natbib.sty
\makeatother% @ becomes a symbol again
\usepackage{hyperref}
\hypersetup{
colorlinks=true,
linkcolor=blue, % Couleur des liens internes
citecolor=red, % Couleur des num?ros de la biblio dans le corps
urlcolor=blue  } % Couleur des url
\usepackage{mathptmx}   %Times New Roman

\theoremstyle{plain}% Theorem-like structures provided by amsthm.sty
\newtheorem{theorem}{Theorem}[section]
\newtheorem{lemma}[theorem]{Lemma}
\newtheorem{corollary}[theorem]{Corollary}
\newtheorem{proposition}[theorem]{Proposition}

\theoremstyle{definition}
\newtheorem{definition}[theorem]{Definition}
\newtheorem{example}[theorem]{Example}

\theoremstyle{remark}
\newtheorem{remark}[theorem]{Remark}

\renewcommand {\theenumi} {\rm\roman{enumi}}

\usepackage{color}
\usepackage{comment}
\newcommand{\bx}{\bar x}
\newcommand{\eps}{\varepsilon}
\newcommand {\B} {\mathbb B}
\def\es{\emptyset}
\def\Fr{Fr\'echet}
\def\EVP{Ekeland variational principle}
\newcommand{\vertiii}[1]{\left\vert\kern-0.25ex\left\vert\kern-0.25ex\left\vert #1\right\vert\kern-0.25ex\right\vert\kern-0.25ex\right\vert}
\newcommand{\qdtx}[1]{\quad\mbox{#1}\quad}
\newcommand{\folgt}{$\Rightarrow$}
\newcommand {\R} {\mathbb R}
\newcommand {\N} {\mathbb N}
\newcommand{\iv}{^{-1} }
\newcommand {\diam} {{\rm diam}\,}
\newcommand{\ang}[1]{\left\langle #1 \right\rangle}
\newcommand {\dom} {{\rm dom}\,}
\newcommand {\epi} {{\rm epi}\,}
\newcommand {\sd} {\partial}
\newcommand{\de}{\delta}
\newcommand{\al}{\alpha}
\newcommand{\be}{\beta}
\newcommand{\ga}{\gamma}
\newcommand{\la}{\lambda}
\renewcommand{\iff}{$\Leftrightarrow$}%iff
\newcommand{\AND}{\quad\mbox{and}\quad}
\newcommand{\abs}[1]{\left\vert#1\right\vert}
\newcommand{\norm}[1]{\left\Vert#1\right\Vert}
\def\lsc{lower semicontinuous}
\newcommand {\Int} {{\rm int}\,}

\newcommand{\NDC}[1]{\todo[inline,color=green!40]{NDC {#1}}}

\begin{document}

\title{
Sequential Extremal Principle and Necessary Conditions for Minimizing Sequences}

\author{
\name{Nguyen Duy Cuong\textsuperscript{a}, Alexander Y. Kruger\textsuperscript{b}}
\thanks{CONTACT Alexander Y. Kruger. Email: alexanderkruger@tdtu.edu.vn}
\thanks{Dedicated to Prof Michel Th\'era, a great scholar and friend}
\affil{\textsuperscript{a} Department of Mathematics, College of Natural Sciences, Can Tho University, Can Tho {City}, Vietnam;
\textsuperscript{b}~Optimization Research Group,
Faculty of Mathematics and Statistics,
Ton Duc Thang University, Ho Chi Minh City, Vietnam}
}

\maketitle

\begin{abstract}
The conventional definition of extremality of a finite collection of sets
%and the corresponding extremal principle
is extended by replacing a fixed point (extremal point) in the intersection of the sets by a collection of sequences of points in the individual sets with the distances between the corresponding points tending to zero.
This allows one to consider collections of unbounded sets with empty intersection.
{Building on the ideas of} the conventional extremal principle, we derive an extended sequential version of the latter result in terms of Fr\'echet and Clarke normals.
Sequential versions of the related concepts of stationarity, approximate stationarity and transversality of collections of sets are also studied.
As an application, we establish sequential necessary conditions for minimizing (and more general firmly stationary, stationary and approximately stationary) sequences in a constrained optimization problem.

%The paper studies extremality, stationarity, and approximate stationarity/transversality properties of collections of sets with respect to sequences. The conventional properties defined at a point are particular cases of the sequential ones. Several examples are provided to illustrate these new properties. Dual necessary conditions for the sequential properties are formulated in terms of Fréchet and Clarke normal cones. As an application, we establish sequential necessary optimality and stationarity conditions for an optimization problem.
\end{abstract}


\begin{amscode}
49J52; 49J53; 49K40; 90C30; 90C46
\end{amscode}

\setcounter{tocdepth}{2}
%\tableofcontents

\section{Introduction}

We continue studying extremality and stationarity properties of collections of sets and the corresponding generalized separation statements in the sense of the \emph{extremal principle} \cite{KruMor80,MorSha96,Mor06.1}.
Such statements are widely used in variational analysis and naturally translate into necessary
optimality conditions (multiplier rules) and various subdifferential and coderivative calculus results in nonconvex settings \cite{KruMor80,Kru85.1,MorSha96,Kru03,Mor06.1,Mor06.2,BorZhu05}.

Throughout the paper we consider a collection of $n>1$ arbitrary nonempty subsets $\Omega_1,\ldots,\Omega_n$ of a normed space $(X,\|\cdot\|)$ and write $\{\Omega_1,\ldots,\Omega_n\}$ to denote the collection of these sets as a single object.

The conventional extremal principle gives necessary conditions for the \emph{extremality} of a collection of sets.

\begin{definition}
[Extremality]
\label{D1.1}
The collection $\{\Omega_1,\ldots,\Omega_n\}$
is extremal at $\bx\in\bigcap_{i=1}^n\Omega_i$ if there is a $\rho\in(0,+\infty]$ such that,
for any $\varepsilon>0$, there exist $a_1,\ldots,a_n\in\eps\B_X$ such that
\begin{gather}
%\label{D1.1-1}
%\|a_i\|<\varepsilon\quad (i=1,\ldots,n),
%\\
\label{D1.1-2}
\bigcap_{i=1}^n(\Omega_i-a_i)\cap B_\rho(\bx)=\emptyset.
\end{gather}
\end{definition}

Here, symbols $\B_X$ and $B_\rho(\bx)$ denote the open unit ball in $X$ and open ball with centre $\bx$ and radius $\rho$, respectively.
For brevity, we combine in the above definition the cases of local ($\rho<+\infty$) and global ($\rho=+\infty$) extremality.
In the latter case, the point $\bx$ plays no role apart from ensuring that $\bigcap_{i=1}^n\Omega_i\ne\es$.
%It is easy to see that the property does not change if $B_\rho(\bx)$ is replaced by the corresponding
%closed ball $\overline B_\rho(\bx)$.

The extremal principle gives necessary conditions for extremality in terms of elements of the (topologically) dual space $X^*$, and assumes that $X$ is Asplund and the sets are closed.

\begin{theorem}
[Extremal principle]
\label{EP}
Let $X$ be Asplund, and $\Omega_1,\ldots,\Omega_n$ %$(n>1)$
be closed.
If $\{\Omega_1,\ldots,\Omega_n\}$
is extremal at $\bx\in\bigcap_{i=1}^n\Omega_i$, then,
for any $\varepsilon>0$, there exist $x_i\in\Omega_i\cap B_\eps(\bx)$ and $x_i^*\in N^F_{\Omega_i}(x_i)$
$(i=1,\ldots,n)$ such that
\begin{gather}
\label{EP-3}
\Big\|\sum_{i=1}^nx_i^*\Big\|<\eps,\quad
\sum_{i=1}^{n}\|x_i^*\|=1.
%\\
%\notag
%\label{EP-2}
%d(x_i^*,N^F_{\Omega_i}(x_i))<\eps\quad
%(i=1,\ldots,n).
\end{gather}
\end{theorem}

Theorem~\ref{EP} employs \emph{\Fr\ normal cones}
(see definition~\eqref{NC}).
The dual conditions are formulated in a \emph{fuzzy} form and can be interpreted as \emph{generalized separation}; cf. \cite{BuiKru19}.
The statement naturally yields the limiting version of the extremal principle in terms of \emph{limiting normal cones} \cite{KruMor80,Kru85.1,MorSha96,Kru03, Mor06.1}.
In finite dimensions, the limiting version comes for free, while in infinite dimensions, additional \emph{sequential normal compactness} type assumptions are required to ensure that the equality in \eqref{EP-3} is preserved when passing to limits (as $\eps\downarrow0$); cf. \cite{Mor06.1}.

The conventional extremal principle was established in this form in \cite{MorSha96} as an extension of the original result from \cite{KruMor80}, which had been formulated in \emph{\Fr\ smooth} spaces and referred to as the \emph{generalized Euler equation}.
The proof employs
%two fundamental results of variational analysis:
the \emph{\EVP} \cite{Eke74} and fuzzy \emph{\Fr\ subdifferential sum rule} due to Fabian \cite{Fab89}.
%see Lemmas~\ref{EVP} and \ref{SR}\,\eqref{SR.Fr}.
It was also shown in \cite{MorSha96} that the necessary conditions in Theorem~\ref{EP} are equivalent to the Asplund property of the space (see also \cite[Theorem~2.20]{Mor06.1}).
Recall that a Banach space is {Asplund} if every continuous convex function on an open convex set is Fr\'echet differentiable on a dense subset, or equivalently, if the dual of each its separable subspace is separable.
We refer the reader to \cite{Phe93,Mor06.1,BorZhu05} for discussions about and characterizations of Asplund spaces.
All reflexive, particularly, all finite dimensional Banach spaces are Asplund.

Definition~\ref{D1.1} and Theorem~\ref{EP} can be reformulated in the setting of the product space $X^n$ for the ``aggregate'' set
$\widehat\Omega:=\Omega_1\times\ldots\times\Omega_n$ employing the maximum norm on $X^n$ and the corresponding dual (sum) norm:
\begin{gather}
\label{norm}
%X^n\ni(u_1,\ldots,u_{n})\mapsto\max_{1\le i\le n} \|u_i\|,
%\quad
%(X^*)^n\ni(u_1^*,\ldots,u_{n}^*)\mapsto\sum_{1=1}^n \|u_i^*\|.
\vertiii{(u_1,\ldots,u_{n})}:=\max_{1\le i\le n} \|u_i\|\qdtx{for all}u_1,\ldots,u_{n}\in X,
\\
\label{normd}
\vertiii{(u_1^*,\ldots,u_{n}^*)}:=\sum_{1=1}^n \|u_i^*\|\qdtx{for all}u_1^*,\ldots,u_{n}^*\in X^*.
\end{gather}
It was observed in \cite{CuoKru25} that Theorem~\ref{EP} and its numerous generalizations and extensions remain valid with an arbitrary product norm $\vertiii{\,\cdot\,}$ (together with the corresponding dual norm) satisfying natural compatibility conditions with the original norm $\|\cdot\|$ on $X$.
This more general setting, in particular, allows one to recapture the \emph{unified separation theorem} of Zheng \& Ng \cite{ZheNg05.2,ZheNg11} employing the \emph{$p$-weighted nonintersect index}, and its slightly more advanced version in \cite
%[Lemma~2.1]
{CuoKruTha24}.

In this paper, following the scheme initiated in \cite{CuoKru25}, we assume that $X^n$ is equipped with a norm $\vertiii{\,\cdot\,}$ satisfying
the following compatibility conditions:
\begin{gather}
\label{C}
\kappa_1\max_{1\le i\le n}\|u_i\|
\le\vertiii{(u_1,\ldots,u_n)} \le\kappa_2\max_{1\le i\le n}\|u_i\|\quad
\text{for all}\;\; u_1,\ldots,u_n\in X
\end{gather}
with some $\kappa_1>0$ and $\kappa_2>0$.
\begin{comment}
\if{\AK{21/01/25.
It seems we do not need an abstract product norm in this paper.
It would suffice to consider just the conventional norm \eqref{norm}.}
\NDC{26/01/25.
I agree that we do not use the extremal principle concerning general norms in the application part.
But there are some advantages of the current form:
1) consistency with our current results; 2) simplifying the presentation; 3) we may need this form when considering more general optimization problems.
It is just my point of view, please feel free to change the presentation if you think it would improve the paper.}
}\fi
\end{comment}
For a discussion of weaker compatibility conditions and some results without such conditions we refer the reader to \cite{CuoKru25}.
Clearly, the norm \eqref{norm} satisfies conditions \eqref{C}.
Under \eqref{C}, if $X$ is Banach/Asplund, then so is $X^n$, and the (\Fr\ or Clarke) normal cone to $\widehat\Omega$ equals the cartesian product of the corresponding cones to the individual sets; see \cite{CuoKru25}.

The next definition contains abstract product norm extensions of the extremality property in Definition~\ref{D1.1} and corresponding stationarity properties studied in \cite{Kru98,Kru02,Kru03,Kru04,Kru05,Kru06,Kru09,BuiKru18}.
Given a $u\in X$, we write $(u,\ldots,u)_n$ to specify that $(u,\ldots,u)\in X^n$.

\begin{definition}
[Extremality, stationarity and approximate stationarity]
\label{D1.3}
Let $\bx\in\bigcap_{i=1}^n\Omega_i$.
The collection $\{\Omega_1,\ldots,\Omega_n\}$ is \begin{enumerate}
\item
\label{D2.5.1}
{extremal} at $\bx$ if there is a $\rho\in(0,+\infty]$ such that,
for any $\varepsilon>0$, there exists a point $(a_1,\ldots,a_n)\in\eps\B_{X^n}$ such that condition \eqref{D1.1-2} is satisfied;
\item
\label{D2.5.2}
{stationary} at $\bx$ if,
for any $\eps>0$,
there exist a $\rho\in(0,\eps)$
and a point $(a_1,\ldots,a_n)\in\eps\rho\B_{X^n}$ such that condition \eqref{D1.1-2} is satisfied;
\item
\label{D2.5.3}
{approximately stationary} at $\bx$ if,
for any $\eps>0$,
there exist a $\rho\in(0,\eps)$, and points $(x_1,\ldots,x_n)\in\widehat\Omega\cap B_\eps((\bx,\ldots,\bx)_n)$ and $(a_1,\ldots,a_n)\in\eps\rho\B_{X^n}$ such that
\begin{gather}
\label{D2.5.3-1}
\bigcap_{i=1}^n(\Omega_i-x_i-a_i)\cap(\rho\B)
=\emptyset.
\end{gather}
\end{enumerate}
\end{definition}

The relationships between the properties in Definition~\ref{D1.3} are straightforward:
\eqref{D2.5.1} \folgt\ \eqref{D2.5.2} \folgt\ \eqref{D2.5.3}.
The approximate stationarity, the weakest of the three properties, is still sufficient for
the generalized separation in Theorem~\ref{EP} in Asplund spaces.
The two properties are actually equivalent.
The next statement from \cite{CuoKru25} generalizes the \emph{extended extremal principle} \cite[Theorem~3.7]{Kru03}.

\begin{theorem}
[Extended extremal principle]
\label{T1.4}
Let $X$ be Asplund, $\Omega_1,\ldots,\Omega_n$ %$(n>1)$
be closed, and $\bx\in\bigcap_{i=1}^n\Omega_i$.
The collection $\{\Omega_1,\ldots,\Omega_n\}$ is  approximately stationary at $\bx$ if and only if,
for any $\varepsilon>0$, there exist $x_i\in\Omega_i\cap B_\eps(\bx)$ and $x_i^*\in N^F_{\Omega_i}(x_i)$ $(i=1,\ldots,n)$ such that
\begin{gather}
\label{T1.4-2}
\Big\|\sum_{i=1}^nx_i^*\Big\|<\eps,\quad
\vertiii{\hat x^*}=1,
\end{gather}
where $\hat x^*:=(x_1^*,\ldots,x_n^*)$.
\end{theorem}

The conventional concept of extremality of a collection of sets in Definition~\ref{D1.1} as well as its generalizations and extensions in Definition~\ref{D1.3} and corresponding characterizations in Theorems~\ref{EP} and \ref{T1.4} are attached to a fixed point in the intersection of the sets (often referred to as \emph{extremal point}).
{This typically corresponds to an optimal or stationary point in an optimization problem, and Theorems~\ref{EP} and \ref{T1.4}, when applied to appropriate collections of sets, give corresponding necessary optimality or stationarity conditions.
This was the main motivation behind the original extremal principle \cite{KruMor80}.
Despite its recognized versatility and numerous applications, this model does not cover an important class of problems involving unbounded sets.
Such situations occur, for instance, when the infimum of a minimized function is not attained.
Nevertheless, some analogues of the extremality, stationarity and generalized separation properties may still hold ``approximately'' with a common point of the sets replaced by unbounded sequences.}
This can be illustrated by the following example of a pair of unbounded sets in $\R^2$.

\begin{example}\label{E1.5}
Let $X:=\R^2$.
Consider closed convex sets $\Omega_1:=\{(x,y)\mid x>0,\;xy\ge1\}$ and
$\Omega_2:=\{(x,y)\mid y\le0\}$; see Figure~\ref{fig1}.
We have $d(\Omega_1,\Omega_2)=0$ while $\Omega_1\cap\Omega_2=\es$.
At the same time, for any $\eps>0$, $t>\eps\iv$ and $\xi\in(\frac1t,\eps)$, taking $x_1:=(t,\frac1t)\in\Omega_1$, $x_2:=(t,0)\in\Omega_2$, $a_1:=(0,-\xi)$ and $a_2:=(0,0)$, one has $(\Omega_1-x_1-a_1)\cap(\Omega_2-x_2-a_2)=\es$.
This can be interpreted as an approximate version of the conditions in Definition~\ref{D1.1} with $\rho=+\infty$ and the pair of points $x_1\in\Omega_1$ and $x_2\in\Omega_2$ (depending on $t$) replacing the nonexistent common point $\bx$.
Note that $\|x_1\|\to+\infty$, $\|x_2\|\to+\infty$ and $\|x_1-x_2\|\to0$ as $t\to+\infty$.

Moreover, assuming for simplicity that $\R^2$ is equipped with the maximum norm, for any $\eps>0$, taking a $t>\max\{\frac1{\sqrt{2\eps}},1\}$, and points $x_1:=(t,\frac1t)\in\Omega_1$ and
$x_2:=(t,0)\in\Omega_2$ as above, we have
$x_1^*:=(-\frac1{2t^2},-\frac12)\in N_{\Omega_1}(x_1)$,
$x_2^*:=(0,\frac12)\in N_{\Omega_2}(x_2)$,
$\|x_1^*\|=\|x_2^*\|=\frac12$ and $\|x_1^*+x_2^*\|=\frac1{2t^2}<\eps$, i.e., the conclusions of Theorem~\ref{EP} hold approximately, with the pair of points $x_1\in\Omega_1$ and $x_2\in\Omega_2$ replacing the nonexistent common point $\bx$.
\begin{figure}[H]
\centering
\includegraphics[width=\linewidth]{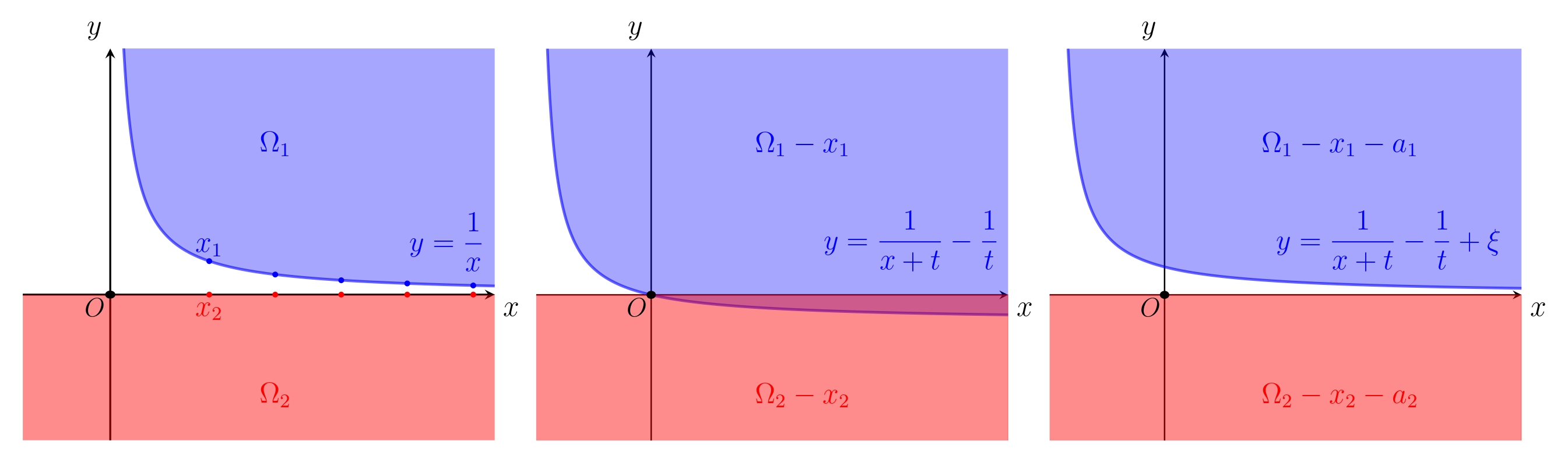}
\caption{Example~\ref{E1.5}}
%Nonempty behavior of $\Omega_1$ and $\Omega_2$ around the origin after linear translations.}
\label{fig1}
\end{figure}
\end{example}

\if{
\AK{26/01/25.
$\Omega_2$ includes also points with negative $x$.

What is the line connecting the origin and $x_1$ for?

Should the points $x_1$ and $x_2$ be further away from the origin?
(We target ``infinity''.)

Should the origin be labelled by zero instead of O?

Is it possible/difficult to depict also $\Omega_1-x_1-a_1$?
}
\NDC{4/2/25.
To be modified.}
}\fi

This paper, motivated by the recent research targeting optimality conditions and subdifferential/coderivative calculus ``at infinity'' in
%\cite{KimTunSon25,kimsontuntuy23,TunSon24,TunSon25},
\cite{NguPha24,KimNguPha},
\if{
\AK{16/02/25.
Should any specific definitions or results from \cite{KimTunSon25,kimsontuntuy23} or any other their papers be mentioned in our paper somewhere?}
\NDC{18/2/25.
I have added two more of their papers.
For this paper, I think citing their works as references is sufficient. In our next paper on limiting objects, there will be relevant definitions/results to discuss.}
}\fi
extends the model in Definition~\ref{D1.3} to cover the case of a collection of unbounded sets
%having ``approximately nonempty'' intersection in the sense that
satisfying
\begin{equation}
\label{d}
d(\Omega_1,\ldots,\Omega_n) :=\inf_{x_1\in\Omega_1,\ldots,x_n\in\Omega_n} \diam\{x_1,\ldots,x_n\}=0,
\end{equation}
where
$\diam\{x_1,\ldots,x_n\}:=
\max_{1\le i,j\le n}\|x_i-x_j\|$ is the diameter of the set $\{x_1,\ldots,x_n\}$.
{As demonstrated in \cite{NguPha24,KimNguPha}, such collections appear naturally in optimization, particularly when the infimum of an objective function is not attained.
The pair of sets in Example~\ref{E1.5} corresponds to the problem of minimizing the function $x\mapsto1/x$ over the half-line $x>0$.
The conventional extremal principle and its ``at a point'' extensions are not applicable in such situations.}

{We replace the common point of the sets $\Omega_1,\ldots,\Omega_n$ (which may not exist) in
the corresponding definitions by appropriate sequences.
Dual necessary conditions in the form of generalized separation extending Theorems~\ref{EP} and \ref{T1.4} and other results are established using standard techniques based on the Ekeland variational principle and subdifferential sum rules.
These standard techniques lying at the core of the conventional extremal principle were encapsulated by Zheng \& Ng in two abstract \emph{non-inter\-section lemmas} \cite{ZheNg05.2} and then in a \emph{unified separation theorem} \cite{ZheNg11} not related to any point and, at the same time, implying the conventional extremal principle and other `at a point' results.
This approach has been further refined in \cite{BuiKru18,BuiKru19,CuoKruTha24,CuoKru25}.
In this paper, we employ a theorem from \cite{CuoKru25} to proving a sequential extremal principle and some related results.}

{The sequential generalized separation results developed in this paper extend the applicability of the conventional approach based on the extremal principle (convincingly promoted in \cite{Mor06.1}), e.g., to optimization problems in which the optimal value is not attained.
Our sequential model seems more flexible than the one developed in \cite{NguPha24,KimNguPha} and does not require introducing an artificial infinity point.}

The new extended model is discussed in Section~\ref{S2}, where sequential versions of the extremality, stationarity and approximate stationarity concepts for a finite collection of sets in a normed vector space are introduced.
The corresponding generalized separation conditions including the sequential extremal principle are established in Section~\ref{S3}.
To illustrate the model, we consider in Section~\ref{S4} a constrained optimization problem and discuss sequential minimality and stationarity properties.
Employing the sequential extremal principle, we deduce  in Section~\ref{S5} sequential optimality and stationarity conditions for the considered constrained optimization problem.
The final Section~\ref{Conclusions} summarises the contributions of the paper and lists potential directions of future research.

\subsubsection*{Preliminaries}
\label{Prel}

Our basic notation is standard, see, e.g., \cite{Mor06.1,BorZhu05}.
The topological dual of a normed space $X$ is denoted by $X^*$, while $\langle\cdot,\cdot\rangle$ denotes the bilinear form defining the pairing between the two spaces.
Symbols $\B$ (possibly with a subscript indicating the space) and $B_\rho(\bx)$ denote the open unit ball and open ball with centre $\bx$ and radius $\rho$, respectively, while $\overline\B$ denotes the closed unit ball.
If $(x,y)\in X\times Y$, we write $B_\varepsilon(x,y)$ instead of $B_\varepsilon((x,y))$.
Symbols $\R$, $\R_+$ and $\N$ stand for the sets of all, respectively, real, nonnegative real and positive integer numbers.
The notation
$\{x^k\}\subset\Omega$ denotes a sequence of points $x^k\in\Omega$ $(k\in\N)$.

We consider the normed spaces $(X,\|\cdot\|)$ and $(X^n,\vertiii{\,\cdot\,})$ with the norm compatibility condition \eqref{C}, and the ``aggregate'' set
$\widehat\Omega:=\Omega_1\times\ldots\times\Omega_n$.
As the meaning will always be clear from the context,
we keep the same notations $\|\cdot\|$ and $\vertiii{\,\cdot\,}$ for the corresponding norms on $X^*$ and $(X^*)^n$, and use $d(\cdot,\cdot)$ to denote
distances (including point-to-set and set-to-set distances) in all spaces determined by the corresponding norms.

\paragraph*{Normal cones and subdifferentials.}
We first recall the definitions of normal cones and subdifferentials in the sense of Fr\'echet and Clarke;
see, e.g., \cite{Cla83,Kru03,Mor06.1}.
Given a subset $\Omega$ of a normed space $X$ and a point $\bx\in \Omega$, the sets
\begin{gather}\label{NC}
N_{\Omega}^F(\bx):= \Big\{x^\ast\in X^\ast\mid
\limsup_{\Omega\ni x{\rightarrow}\bar x,\;x\ne\bx} \frac {\langle x^\ast,x-\bx\rangle}
{\|x-\bx\|} \le 0 \Big\},
\\\label{NCC}
N_{\Omega}^C(\bx):= \left\{x^\ast\in X^\ast\mid
\ang{x^\ast,z}\le0
\;\;
\text{for all}
\;\;
z\in T_{\Omega}^C(\bx)\right\}
\end{gather}
are the, respectively, \emph{Fr\'echet} and \emph{Clarke normal cones} to $\Omega$ at $\bx$.
Symbol $T_{\Omega}^C(\bx)$
in \eqref{NCC}
stands for the \emph{Clarke tangent cone} to $\Omega$ at $\bx$:
\begin{multline*}%\label{TCC}
T_{\Omega}^C(\bx):= \big\{z\in X\mid
\forall x_k{\rightarrow}\bx,\;x_k\in\Omega,\;\forall t_k\downarrow0,\;\exists z_k\to z
\;\;
\text{with}
\;\;
x_k+t_kz_k\in \Omega
\;\;
\text{for all}
\;\;
k\in\N\big\}.
\end{multline*}
%The Clarke tangent cone is always convex and contains 0.
The sets \eqref{NC} and \eqref{NCC} are nonempty
closed convex cones satisfying $N_{\Omega}^F(\bx)\subset N_{\Omega}^C(\bx)$.
If $\Omega$ is a convex set, they reduce to the normal cone $N_{\Omega}(\bx)$ in the sense of convex analysis.

For an extended-real-valued function $f:X\to\R_\infty:=\R\cup\{+\infty\}$ on a normed space $X$,
its domain and epigraph are defined,
respectively, by
$\dom f:=\{x \in X\mid f(x) < +\infty\}$
and
$\epi f:=\{(x,\alpha) \in X \times \mathbb{R}\mid f(x) \le \alpha\}$.
The \emph{Fr\'echet} and \emph{Clarke subdifferentials} of $f$ at $\bar x\in\dom f$
%(cf. \cite{Kru03,Cla83})
are defined, respectively, by
\begin{gather}\label{SF}
\sd^F f(\bar x):= \left\{x^* \in X^*\mid (x^*,-1) \in N^F_{\epi f}(\bar x,f(\bar x))\right\},
\\ \label{SC}
\partial^C{f}(\bx):= \left\{x^\ast\in X^\ast\mid
(x^*,-1)\in N_{\epi f}^C(\bx,f(\bx))\right\}.
\end{gather}
The sets \eqref{SF} and \eqref{SC} are closed and convex, and satisfy
$\partial^F{f}(\bx)\subset\partial^C{f}(\bx)$.
If $f$ is convex, they
reduce to the subdifferential $\sd f(\bx)$ in the sense of convex analysis.

\paragraph*{Generalized separation.}

The next generalized separation statement is the key tool in the proof of our main result.
It is a simplified local version of a more general separation statement from \cite
%[Theorem~3.1]
{CuoKru25}.
\if{
\AK{12/12/24.
Avoid referring to particular statements/definitions and formulas in unpublished papers.
The numbering can change, and you can forget to update it, or it can be too late.}
}\fi
\begin{theorem}
[Generalized separation]
\label{C3.9}
Let $X$ be Banach, $\Omega_1,\ldots,\Omega_n$ %$(n>1)$
be closed,
$\widehat\omega\in\widehat\Omega$,
$\hat x^\circ:= (x^\circ_1,\ldots,x^\circ_n)\in X^n$, $\eps>0$, $\de>0$ and $\rho>0$.
Suppose that  ${\bigcap}_{i=1}^n(\Omega_i-x^\circ_i)\cap (\rho\overline\B)=\emptyset$, and
$\vertiii{\widehat\omega-\hat x^\circ}<\varepsilon$.
The following assertions hold true:
\begin{enumerate}
\item
\label{C3.3-1}
there exist points $\hat x:=(x_1,\ldots,x_n)\in\widehat\Omega\cap B_\de(\widehat\omega)$, {$x_0\in\rho\B$} and $\hat x^*:=(x_1^*,\ldots,x_n^*)\in (X^*)^n$ such that $\vertiii{\hat x^*}=1$,
and
\begin{gather*}
{\de}d\left(\hat x^*,N^C_{\widehat\Omega}(\hat x)\right)+ \rho\Big\|\sum_{i=1}^nx_i^*\Big\| <{\eps},\\
\ang{\hat x^*,(x_0,\dots,x_0)_n+\hat x^\circ-\hat x}= \vertiii{(x_0,\dots,x_0)_n+\hat x^\circ-\hat x};
\end{gather*}
\item
\label{C3.3-2}
if $X$ is Asplund, then, for any $\tau\in(0,1)$, there exist points $\hat x:=(x_1,\ldots,x_n)\in\widehat\Omega\cap B_\de(\widehat\omega)$, {$x_0\in\rho\B$} and $\hat x^*:=(x_1^*,\ldots,x_n^*)\in (X^*)^n$ such that $\vertiii{\hat x^*}=1$, and
\begin{gather*}
{\de}d\left(\hat x^*,N^F_{\widehat\Omega}(\hat x)\right)+ \rho\Big\|\sum_{i=1}^nx_i^*\Big\| <{\eps},\\
\ang{\hat x^*,(x_0,\dots,x_0)_n+\hat x^\circ-\hat x}>\tau\vertiii{(x_0,\dots,x_0)_n+\hat x^\circ-\hat x}.
\end{gather*}
\end{enumerate}
\end{theorem}

Similar to the conventional extremal principle, the latter statement is a consequence of the {\EVP} and corresponding subdifferential sum rules.
With the appropriate product space norms, it
covers the {unified separation theorems} by Zheng \& Ng \cite{ZheNg05.2,ZheNg11} and their slightly more advanced versions in \cite
%[Lemma~2.1]
{CuoKruTha24}.
Theorem~\ref{C3.9} combines two assertions: the traditional Asplund space one covering Theorems~\ref{EP} and \ref{T1.4} in part \eqref{C3.3-2} and the general Banach space assertion in terms of {Clarke normal cones} in part \eqref{C3.3-1}.

\begin{remark}
Theorem~\ref{C3.9} remains true if the assumption $\vertiii{\widehat\omega-\hat x^\circ}<\varepsilon$ is replaced by the next weaker one (see \cite{CuoKru25}):
$\vertiii{\widehat\omega-\hat x^\circ} <\inf_{\substack{\hat u\in\widehat\Omega,\; u\in\rho\B}} \vertiii{\hat u-\hat x^\circ -(u,\ldots,u)_n}+\varepsilon$.
\end{remark}

\section{Sequential extremality, stationarity and approximate stationarity}
\label{S2}

In this section, we discuss sequential versions of the extremality, stationarity and approximate stationarity concepts for a finite collection of sets in a normed vector space.

In what follows, the sets $\Omega_1,\ldots,\Omega_n$ are not supposed to have a common point.
Instead, we assume
%their intersection to be ``approximately nonempty'' in the sense of
that $d(\Omega_1,\ldots,\Omega_n)=0$ (see
\eqref{d}), i.e., there exist sequences
$\{x_i^k\}\subset\Omega_i$ $(i=1,\ldots,n)$ such that $\diam\{x_1^k,\ldots,x_n^k\}\to0$ as $k\to+\infty$.
The sequences may be unbounded.

The next definition is a modification of Definition~\ref{D1.3} employing sequences of the type described above.

\begin{definition}
[Sequential extremality, stationarity and approximate stationarity]
\label{D3.1}
The collection $\{\Omega_1,\ldots,\Omega_n\}$
is
\begin{enumerate}
\item
\label{D3.1.1}
extremal at sequences
$\{x_i^k\}\subset\Omega_i$ $(i=1,\ldots,n)$ if $\diam\{x_1^k,\ldots,x_n^k\}\to0$ as $k\to+\infty$, and there is a $\rho\in(0,+\infty]$ such that,
for any $\varepsilon>0$, there exist an integer $k>\eps\iv$ and a point
$(a_1,\ldots,a_n)\in\eps\B_{X^n}$ such that
\begin{gather}
\label{D1.3-1}
\bigcap_{i=1}^n(\Omega_i-x_i^k-a_i)\cap(\rho\B)=\emptyset;
\end{gather}
%where $x_i:=x_i^k$ $(i=1,\ldots,n)$;
\item
\label{D3.1.2}
stationary at sequences
$\{x_i^k\}\subset\Omega_i$ $(i=1,\ldots,n)$ if $\diam\{x_1^k,\ldots,x_n^k\}\to0$ as $k\to+\infty$, and, for any $\varepsilon>0$, there exist an integer $k>\eps\iv$, a $\rho\in(0,\varepsilon)$, and a point
$(a_1,\ldots,a_n)\in\eps\rho\B_{X^n}$ such that
condition \eqref{D1.3-1} is satisfied;
\item
\label{D3.1.3}
approximately stationary at a sequence $\{x^k\}\subset X$ if, for any $\varepsilon>0$, there exist an integer $k>\eps\iv$, a $\rho\in(0,\varepsilon)$,
and points $(x_1,\ldots,x_n)\in\widehat\Omega\cap B_\eps((x^k,\ldots,x^k)_n)$ and  $(a_1,\ldots,a_n)\in\eps\rho\B_{X^n}$ such that condition \eqref{D2.5.3-1} is satisfied.
\end{enumerate}
\end{definition}

Definition~\ref{D1.3} is a particular case of Definition~\ref{D3.1} with $x_i^k:=\bx$ for all $i=1,\ldots,n$ and $k\in\N$ in parts \eqref{D3.1.1} and \eqref{D3.1.2}, and $x^k:=\bx$ for all $k\in\N$ in part \eqref{D3.1.3}.

The number $\rho$ in part \eqref{D3.1.1} of Definition~\ref{D3.1} is an important quantitative measure of the extremality property.
In the sequel, if the property holds, we will sometimes specify that $\{\Omega_1,\ldots,\Omega_n\}$ is extremal at sequences $\{x_i^k\}\subset\Omega_i$ $(i=1,\ldots,n)$ with number $\rho$.

{The pairs of sets in the examples below are not extremal in the sense of Definition~\ref{D1.1} as their intersections are empty.
At the same time, they are extremal in the sense of Definition~\ref{D3.1}\,\eqref{D3.1.1} at appropriate sequences.}

\begin{example}
[Sequential extremality]
\label{E3.4}
\begin{enumerate}
\item
{The pair of sets $\Omega_1$ and
$\Omega_2$ in Example~\ref{E1.5} is extremal at the sequences $\{(k,1/k)\}\subset\Omega_1$ and $\{(k,0)\}\subset\Omega_2$.}
\item
\label{E3.4.1}
The pair of closed convex sets $\Omega_1:=\{(x,y)\mid y\ge e^{-x}\}$ and
$\Omega_2:=\{(x,y)\mid y\le0\}$ (see Figure~\ref{fig2}) is extremal at the pair of sequences $\{(k,e^{-k})\}\subset\Omega_1$ and $\{(k,0)\}\subset\Omega_2$.

\item
\label{E3.4.2}
The pair of closed sets $\Omega_1:=\{(x,y)\mid xy\ge1\}$ and $\Omega_2:=\{(x,y)\mid xy\le0\}$ (see Figure~\ref{fig2}) is extremal at the following pairs of sequences: 1) $\{(k,1/k)\}\subset\Omega_1$ and $\{(k,0)\}\subset\Omega_2$; 2) $\{(1/k,k)\}\subset\Omega_1$ and $\{(0,k)\}\subset\Omega_2$; 3) $\{(-k,-1/k)\}\subset\Omega_1$ and $\{(-k,0)\}\subset\Omega_2$; 4) $\{(1/k,-k)\}\subset\Omega_1$ and $\{(0,-k)\}\subset\Omega_2$.
The above four pairs of sequences determine four natural ``extremal directions''.
One can also consider various combinations of the above pairs of sequences, not related to any ``directions'', e.g., $\{((-1)^kk,(-1)^kk\iv)\}\subset\Omega_1$ and $\{((-1)^kk,0)\}\subset\Omega_2$.

\item
\label{E3.4.3}
The pair of closed sets $\Omega_1:=\{(x,y)\mid {x}y\ge{x^2+1}\}$ and $\Omega_2:=\{(x,y)\mid x(y-x)\le0\}$ (see Figure~\ref{fig3}) is extremal at the following pairs of sequences (determining four ``extremal directions''): 1) $\{(k,k+1/k)\}\subset\Omega_1$ and $\{(k,k)\}\subset\Omega_2$; 2)~$\{(1/k,k+1/k)\}\subset\Omega_1$ and $\{(0,k)\}\subset\Omega_2$; 3) $\{(-k,-1/k)\}\subset\Omega_1$ and $\{(-k,-k)\}\subset\Omega_2$; 4)~$\{(-1/k,-k-1/k)\}\subset\Omega_1$ and $\{(0,-k)\}\subset\Omega_2$.

\item
\label{E3.4.4}
The pair of sets $\Omega_1:=\{(x,\sin\frac{1}{x})\mid x\ne 0\}$ and $\Omega_2:=\{(x,\frac{1}{x}+\sin\frac{1}{x})\mid x\ne 0\}$ (see Figure~\ref{fig3}) is extremal at the pairs of sequences: 1)~$\{(k,\sin\frac{1}{k})\}\subset\Omega_1$ and $\{(k,\frac{1}{k}+\sin\frac{1}{k})\}\subset\Omega_2$;
2)~$\{(-k,-\sin\frac{1}{k})\}\subset\Omega_1$ and $\{(-k,-\frac{1}{k}-\sin\frac{1}{k})\}\subset\Omega_2$.
\end{enumerate}
\begin{figure}[H]
\centering
\includegraphics[width=.8\linewidth]{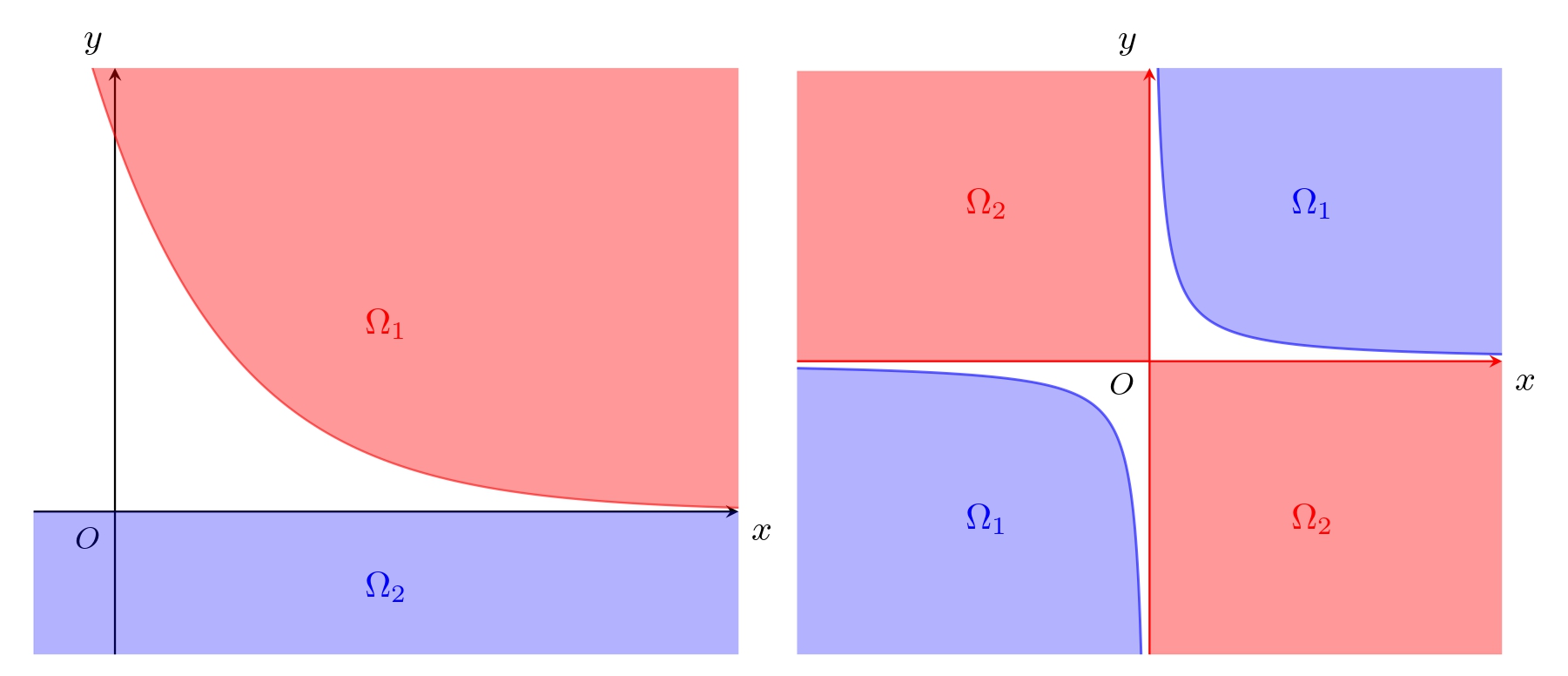}
\\
{\scriptsize
Extremal at sequences $\{(k,e^{-k})\}$ and $\{(k,0)\}$ \quad
Extremal at sequences $\{(1/k,k)\}$ and $\{(0,k)\}$}
\caption{Example~\ref{E3.4}\,\eqref{E3.4.1} and \eqref{E3.4.2}}
\label{fig2}
\end{figure}
\begin{figure}[H]
\centering
\includegraphics[width=.8\linewidth]{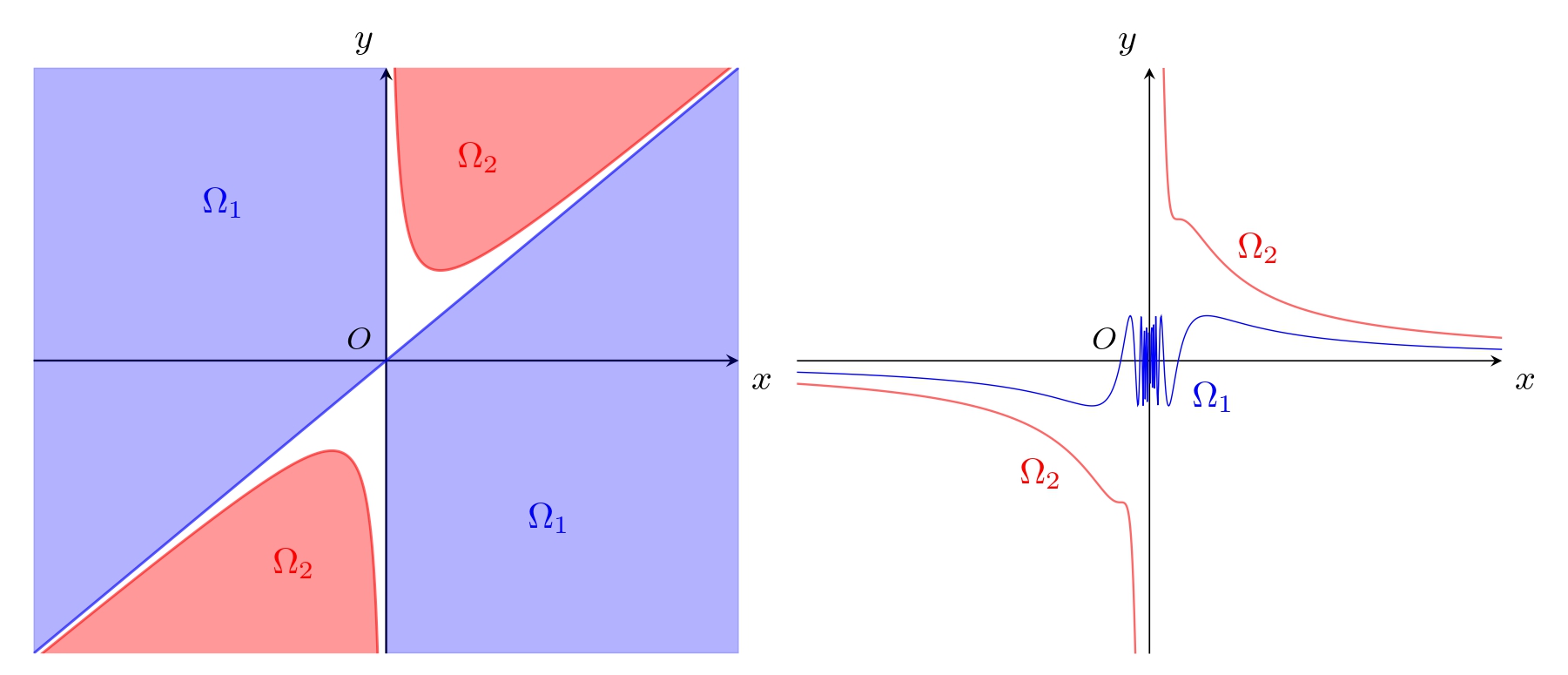}
\\
{\scriptsize
Extremal at sequences $\{(k,k+1/k)\}$ and $\{(k,k)\}$ \quad
Extremal at sequences $\{(k,\sin\frac1k)\}$ and $\{(k,\frac1k+\sin\frac1k)\}$}
\caption{Example~\ref{E3.4}\,\eqref{E3.4.3} and \eqref{E3.4.4}}
\label{fig3}
\end{figure}
\end{example}

\begin{remark}
\label{R02}
\begin{enumerate}
\item
Condition $\diam\{x_1^k,\ldots,x_n^k\}\to0$ as  $k\to+\infty$ in parts \eqref{D3.1.1} and \eqref{D3.1.2} of Definition~\ref{D3.1} is satisfied if all the sequences $\{x_i^k\}$ $(i=1,\ldots,n)$ converge to the same point.
On the other hand, this condition
ensures that, if any of the sequences $\{x_i^k\}$ $(i=1,\ldots,n)$ has a cluster point, then it is a common cluster point of all the sequences; otherwise, $\|x_i^k\|\to+\infty$ for all $i\in\{1,\ldots,n\}$.
\item
{Unlike parts \eqref{D3.1.1} and \eqref{D3.1.2} of Definition~\ref{D3.1} employing $n$ \emph{given} sequences
$\{x_i^k\}\subset\Omega_i$ such that $\diam\{x_1^k,\ldots,x_n^k\}\to0$ as $k\to+\infty$ $(i=1,\ldots,n)$, the property in part \eqref{D3.1.3} uses a single sequence $\{x^k\}$ which needs not lie in any of the sets.
At the same time, the latter definition ensures the existence of \emph{some} sequences $\{x_i^k\}\subset\Omega_i$ such that $\diam\{x_1^k,\ldots,x_n^k\}\to0$ as $k\to+\infty$ $(i=1,\ldots,n)$.}

\item
\label{R02.3}
Similarly to Definition~\ref{D1.3},
it holds \eqref{D3.1.1} $\Rightarrow$ \eqref{D3.1.2} in Definition~\ref{D3.1} and, if $\{\Omega_1,\ldots,\Omega_n\}$ is stationary at sequences
$\{x_i^k\}\subset\Omega_i$ $(i=1,\ldots,n)$, then, for each $i\in\{1,\ldots,n\}$, it is approximately stationary at the sequence $\{x_{i}^k\}$.
The latter claim can be strengthened as shown in Proposition~\ref{P3.4}\,\eqref{P3.4.4}.
%When the sets are convex, all these properties are equivalent, see Proposition~\ref{P3.5}.

\item
{If a property in Definition~\ref{D1.3} holds at some sequences (sequence), then it holds at their subsequences (its subsequence) generated by any sequence of integers $k_j\to+\infty$.}
\end{enumerate}
\end{remark}

The next proposition collects some elementary facts about the properties in Definition~\ref{D3.1}.

\begin{proposition}
\label{P3.4}
\begin{enumerate}
\item
\label{P3.4.1}
If $\{\Omega_1,\ldots,\Omega_n\}$ is
extremal at $\{x_i^k\}\subset\Omega_i$ $(i=1,\ldots,n)$, then it is extremal at any $\{x_i'^k\}\subset\Omega_i$ satisfying $\|x_i'^k-x_i^k\|\to0$ as $k\to+\infty$ $(i=1,\ldots,n)$.
\item
\label{P3.4.2}
If $\{\Omega_1,\ldots,\Omega_n\}$ is
stationary at $\{x_i^k\}\subset\Omega_i$ $(i=1,\ldots,n)$, then it is stationary at any $\{x_i'^k\}\subset\Omega_i$ satisfying $\|x_i'^k-x_i^k\|\to0$ as $k\to+\infty$ $(i=1,\ldots,n)$.
\item
\label{P3.4.3}
If $\{\Omega_1,\ldots,\Omega_n\}$ is
approximately stationary at $\{x^k\}\subset X$, then it is approximately stationary at any $\{x'^k\}\subset X$ satisfying $\|x'^k-x^k\|\to0$ as $k\to+\infty$.
\item
\label{P3.4.4}
Let $\{x_i^k\}\subset\Omega_i$, $\al_i\in\R$ $(i=1,\ldots,n)$, $\sum_{i=1}^n\al_i=1$, and
$x^k:=\sum_{i=1}^{n}\al_i x_i^k$ for all $k\in\N$.
Suppose that $\{\Omega_1,\ldots,\Omega_n\}$
is stationary at
$\{x_i^k\}$ $(i=1,\ldots,n)$.
Then it is approximately stationary at $\{x^k\}$.
\end{enumerate}
\end{proposition}

\begin{proof}
\begin{enumerate}
\item
Let $\{\Omega_1,\ldots,\Omega_n\}$ be
extremal at $\{x_i^k\}\subset\Omega_i$ $(i=1,\ldots,n)$ with some $\rho\in(0,+\infty]$.
Let $\{x_i'^k\}\subset\Omega_i$ and $\|x_i'^k-x_i^k\|\to0$ as $k\to+\infty$ $(i=1,\ldots,n)$.
By Definition~\ref{D3.1}\,\eqref{D3.1.1},
$\diam\{x_1^k,\ldots,x_n^k\}\to 0$ as $k\to+\infty$.
For all $i,j\in\{1,\ldots,n\}$ and $k\in\N$, we have
$
\|x_i'^k-x_j'^k\|
\le \|x_i'^k-x_i^k\|+\|x_i^k-x_j^k\|+\|x_j'^k-x_j^k\|
$.
Hence, $\diam\{x_1'^k,\ldots,x_n'^k\}\to0$ as $k\to+\infty$.

Let $\varepsilon>0$.
Then there is an integer $k>\eps\iv$, such that $\vertiii{(x_1'^k-x_1^k,\ldots,x_n'^k-x_n^k)}<\eps/2$, and a point $(a_1,\ldots,a_n)\in(\eps/2)\B_{X^n}$ such that
condition
\eqref{D1.3-1} is satisfied.
Set $a_i':=a_i+x_i^k-x_i'^k$ $(i=1,\ldots,n)$.
Then $(a_1',\ldots,a_n')\in\eps\B_{X^n}$ and,
in view of condition \eqref{D1.3-1}, we have
$\bigcap_{i=1}^n(\Omega_i-x_i'^k-a_i')\cap (\rho\B)=\es$.
Hence, $\{\Omega_1,\ldots,\Omega_n\}$ is extremal at $\{x_i'^k\}$ $(i=1,\ldots,n)$ (with the same~$\rho$).

\item
The proof of the assertion goes as above.
The fact that $\rho$ is chosen after~$\eps$ does not affect the arguments.

\item
Let $\{\Omega_1,\ldots,\Omega_n\}$ be approximately stationary at $\{x^k\}\subset X$, and $\|x'^k-x^k\|\to0$ as $k\to+\infty$.
Let $\varepsilon>0$.
By Definition~\ref{D3.1}\,\eqref{D3.1.3}, there exist an integer $k>\eps\iv$, a $\rho\in(0,\varepsilon)$, and points $(x_1,\ldots,x_n)\in\widehat\Omega\cap B_{\eps/2}((x_k,\ldots,x_k)_n)$ and $(a_1,\ldots,a_n)\in\eps\rho\B_{X^n}$ such that $\vertiii{(x'^k-x^k,\ldots,x'^k-x^k)_n}
<\varepsilon/2$, and
condition \eqref{D2.5.3-1} is satisfied.
Then
$\vertiii{(x_{1}-x'^k,\ldots,x_{n}-x'^k)}
%\le \vertiii{(x_{1}-x^k,\ldots,x_{n}-x^k)} +\vertiii{(x^k-x'^k,\ldots,x^k-x'^k)_n}
<\varepsilon$.
Hence, $\{\Omega_1,\ldots,\Omega_n\}$ is approximately stationary at~$\{x'^k\}$.
\item
Let $\varepsilon>0$.
By Definition~\ref{D3.1}\,\eqref{D3.1.2},
$\diam\{x_1^k,\ldots,x_n^k\}\to 0$ as $k\to+\infty$, and there exist an integer $k>\eps\iv$, a $\rho\in(0,\varepsilon)$, and a point $(a_1,\ldots,a_n)\in\eps\rho\B_{X^n}$ such that condition \eqref{D1.3-1} is satisfied.
Hence, condition \eqref{D2.5.3-1} holds true with $x_i:=x_i^k$ $(i=1,\ldots,n)$.
Let the second inequality in \eqref{C} be satisfied with some $\kappa_2>0$.
Without loss of generality, we can assume that
$n\kappa_2\max_{1\le j\le n}|\al_j|\cdot\diam\{x_1^k,\ldots,x_n^k\}<\varepsilon$.
For each $i\in\{1,\ldots,n\}$, we have
\begin{align*}
\|x_{i}-x^k\|=\Big\|\sum_{j=1}^n\al_j(x_{i}-x_j^k)\Big\| &\le\sum_{j=1}^n|\al_j|\cdot\|x_{i}-x_j^k\|
\\&\le n\max_{1\le j\le n}|\al_j|\cdot \diam\{x_1^k,\ldots,x_n^k\}<\eps/{\kappa_2}.
\end{align*}
Hence, $\vertiii{(x_{1}-x^k,\ldots,x_{n}-x^k)} \le\kappa_2\max_{1\le i\le n}\|x_{i}-x^k\|<\eps$.
Thus, $\{\Omega_1,\ldots,\Omega_n\}$ is approximately stationary at $\{x^k\}$.
\end{enumerate}
%The last two assertions are straightforward.
\end{proof}
\if{
\AK{24/01/25.
Should the last two assertions be moved into Remark~\ref{R02}?}
}\fi
{As a consequence of Proposition~\ref{P3.4}\,\eqref{P3.4.1}--\eqref{P3.4.3},
if the sequences
in Definition~\ref{D3.1} converge to a point $\bx\in\bigcap_{i=1}^n\Omega_i$, then}
the corresponding properties in Definitions~\ref{D1.3} and \ref{D3.1} are equivalent.

\begin{corollary}
\label{P1.4}
Let $\bx\in\bigcap_{i=1}^n\Omega_i$,
$\Omega_i\ni x_i^k\to\bx$ $(i=1,\ldots,n)$ and $X\ni x^k\to\bx$ as $k\to+\infty$.
The collection $\{\Omega_1,\ldots,\Omega_n\}$ is \begin{enumerate}
\item
\label{P1.4.1}
extremal at $\{x_i^k\}$ $(i=1,\ldots,n)$ if and only if it is extremal at $\bx$;
\item
\label{P1.4.2}
stationary at $\{x_i^k\}$ $(i=1,\ldots,n)$ if and only if it is stationary at $\bx$;
\item
\label{P1.4.3}
approximately stationary at $\{x^k\}$ if and only if it is approximately stationary at $\bx$.
\end{enumerate}
\end{corollary}

\begin{remark}
When $\Omega_1,\ldots,\Omega_n$ are closed, conditions $\Omega_i\ni x_i^k\to\bx$ $(i=1,\ldots,n)$ automatically yield $\bx\in\bigcap_{i=1}^n\Omega_i$.
\end{remark}

Note that Definition \ref{D3.1} does not assume the sequences to be convergent or even bounded.
Unbounded sets and sequences are of special interest in this paper.
An unbounded sequence can define a certain direction, e.g., $\{(k,k)\}\subset\R^2$, but this is not a requirement;
consider, e.g., $\{(k,k^2)\}$ or $\{((-1)^kk,0)\}$.
%or $\{((((-1)^{k+1}+1)k),0)\}$.

In the convex case, the properties in Definition~\ref{D3.1} admit simpler representations and are mostly equivalent.

\begin{proposition}
[Sequential extremality and stationarity: convex case]
\label{P3.5}
Let $\Omega_1,\ldots,\Omega_n$ be convex, 	
$\{x_i^k\}\subset\Omega_i$ $(i=1,\ldots,n)$ and  $\diam\{x_1^k,\ldots,x_n^k\}\to0$ as $k\to+\infty$.
The following assertions are equivalent:
\begin{enumerate}
\item
\label{P3.5.1}
$\{\Omega_1,\ldots,\Omega_n\}$ is extremal at  $\{x^k_i\}$ $(i=1,\ldots,n)$ with any $\rho\in(0,+\infty)$;
\item
\label{P3.5.2}
$\{\Omega_1,\ldots,\Omega_n\}$ is extremal at  $\{x^k_i\}$ $(i=1,\ldots,n)$ with some $\rho\in(0,+\infty)$;
\item
\label{P3.5.3}
$\{\Omega_1,\ldots,\Omega_n\}$ is stationary at  $\{x^k_i\}$ $(i=1,\ldots,n)$;
\item
\label{P3.5.4}
for any $\eps,\rho\in(0,+\infty)$, there exist an integer $k>\eps\iv$, and a point $(a_1,\ldots,a_n)\in\eps\B_{X^n}$ such that condition \eqref{D1.3-1} is satisfied. \end{enumerate}	
\end{proposition}

\begin{proof}
The relations \eqref{P3.5.4} \iff\ \eqref{P3.5.1} $\Rightarrow$ \eqref{P3.5.2} $\Rightarrow$ \eqref{P3.5.3} are direct consequences of the definitions (cf. Remark~\ref{R02}\;\eqref{R02.3}).
We now prove implication \eqref{P3.5.3} $\Rightarrow$ \eqref{P3.5.4}.
Suppose that \eqref{P3.5.4} does not hold, i.e., there exist $\varepsilon,\rho\in(0,+\infty)$ such that
\begin{gather}
\label{P3.5P1}
\bigcap_{i=1}^n(\Omega_i-x_i^k-a_i)\cap (\rho\B)\ne\emptyset
\end{gather}	
for all integers $k>\eps\iv$ and points $(a_1,\ldots,a_n)\in\varepsilon\B_{X^n}$.
Set $\varepsilon':=\min\{\eps,\varepsilon/\rho\}$, and
take arbitrarily an integer $k>\eps'{}\iv$, a positive $\rho'<\varepsilon'$ and a point $(a_1',\ldots,a_n')\in\varepsilon'\rho'\B_{X^n}$.
Set $t:={\rho'}/{\rho}$.
Then $\vertiii{(a_1'/t,\ldots,a_n'/t)}<\varepsilon'\rho\le\eps$, and condition \eqref{P3.5P1} implies the existence of an $x'\in\rho\B$ such that
$x'+x_i^k+a_i'/t\in\Omega_i$ $(i=1,\ldots,n)$.
Then $tx'\in\rho'\B$ and
$tx'+x_i^k+a_i'=t(x'+x_i^k+a_i'/t)+(1-t)x_i^k\in \Omega_i$ $(i=1,\ldots,n)$.
Hence,
$tx'\in\bigcap_{i=1}^n(\Omega_i-x_i^k-a_i')\cap (\rho'\B)$,
and consequently, assertion \eqref{P3.5.3} does not hold.
\end{proof}

\begin{proposition}
[Sequential approximate stationarity: convex case]
\label{P3.6}
Let $\Omega_1,\ldots,\Omega_n$ be convex, 	
and $\{x^k\}\subset X$.
\begin{enumerate}
\item\label{P3.6.1}
$\{\Omega_1,\ldots,\Omega_n\}$ is
approximately sta\-tionary at $\{x^k\}$ if and only if for any $\varepsilon,\rho\in(0,+\infty)$, there exist an integer $k>\eps\iv$, and points $(x_1,\ldots,x_n)\in\widehat\Omega\cap B_\eps((x^k,\ldots,x^k)_n)$ and $(a_1,\ldots,a_n)\in\eps\B_{X^n}$ such that condition \eqref{D2.5.3-1} is satisfied.
\item	
If $\{\Omega_1,\ldots,\Omega_n\}$ is approximately sta\-tionary at $\{x^k\}$, then there exist a subsequence ${\{x^{k_j}\}}$, $j=1,2\ldots$, and sequences $\{x^j_i\}\subset\Omega_i$ $(i=1,\ldots,n)$ such that $x^j_i-x^{k_j}\to0$ as $j\to+\infty$ $(i=1,\ldots,n)$, and $\{\Omega_1,\ldots,\Omega_n\}$ is extremal at $\{x^j_i\}$ $(i=1,\ldots,n)$ with any $\rho>0$.
\end{enumerate}
\end{proposition}	

\begin{proof}
\begin{enumerate}
\item
The ``if'' part is obvious.
We prove the ``only if'' part.
Suppose that there exist $\varepsilon,\rho\in(0,+\infty)$ such that
\begin{gather}
\label{P3.6P1}
\bigcap_{i=1}^n(\Omega_i-x_i-a_i)\cap (\rho\B)\ne\emptyset
\end{gather}	
for all integers $k>\eps\iv$, and points $(x_1,\ldots,x_n)\in\widehat\Omega\cap B_\eps((x^k,\ldots,x^k)_n)$ and $(a_1,\ldots,a_n)\in\varepsilon\B_{X^n}$.
Set $\varepsilon':=\min\{\eps,\varepsilon/\rho\}$, and
take arbitrarily an integer $k>\eps'{}\iv$, a positive ${\rho'<\varepsilon'}$ and points $(x_1,\ldots,x_n)\in\widehat\Omega\cap B_{\eps'}((x^k,\ldots,x^k)_n)$ and $(a_1',\ldots,a_n')\in\varepsilon'\rho'\B_{X^n}$.
Set $t:={\rho'}/{\rho}$.
Then $\vertiii{(a_1'/t,\ldots,a_n'/t)}<\varepsilon'\rho\le\eps$, and condition \eqref{P3.6P1} implies the existence of an $x'\in\rho\B$ such that
$x'+x_i+a_i'/t\in\Omega_i$ $(i=1,\ldots,n)$.
Then $tx'\in\rho'\B$ and
$tx'+x_i+a_i'=t(x'+x_i+a_i'/t)+(1-t)x_i\in \Omega_i$ $(i=1,\ldots,n)$.
Hence,
$tx'\in\bigcap_{i=1}^n(\Omega_i-x_i-a_i')\cap (\rho'\B)$,
and consequently,
$\{\Omega_1,\ldots,\Omega_n\}$ is not
approximately sta\-tionary at $\{x^k\}$.
\item	
Suppose that $\{\Omega_1,\ldots,\Omega_n\}$ is approximately sta\-tionary at $\{x^k\}$.
Let $\eps,\rho\in(0,+\infty)$.
By \eqref{P3.6.1},
for any $j\in\N$, there exist an integer $k_j>j$, and points $(x_1^j,\ldots,x_n^j)\in\widehat\Omega\cap B_{1/j}((x^{k_j},\ldots,x^{k_j})_n)$ and $(a_1^j,\ldots,a_n^j)\in(1/j)\B_{X^n}$ such that \begin{gather}
\label{P3.6P2}
\bigcap_{i=1}^n(\Omega_i-x_i^j-a_i^j)\cap (\rho\B)=\emptyset.
\end{gather}	
Thus, $k_j\to+\infty$, $x^j_i-x^{k_j}\to0$ $(i=1,\ldots,n)$ and $\diam\{x_1^j,\ldots,x_n^j\}\to0$  as $j\to+\infty$.
Take an integer $j>\eps\iv$.
Then $(a_1^j,\ldots,a_n^j)\in\eps\B_{X^n}$ and, in view of \eqref{P3.6P2}, assertion \eqref{P3.5.4} in Proposition~\ref{P3.5} is satisfied.
The conclusion follows from Proposition~\ref{P3.5}.
\end{enumerate}
\end{proof}

\section{Sequential extremal principle}
\label{S3}

In this section, we study a quantitative version of the sequential approximate stationarity property in Definition~\ref{D3.1}\,\eqref{D3.1.3} and prove dual necessary conditions in the form of generalized separation.

\begin{definition}
[Sequential approximate $\al$-stationarity]
\label{D3.5}
Let $\al>0$.
The collection $\{\Omega_1,\ldots,\Omega_n\}$ is approximately $\al$-stationary at a sequence $\{x^k\}\subset X$ if, for any $\varepsilon>0$, there exist an integer $k>\eps\iv$, a $\rho\in(0,\varepsilon)$, and points $(x_1,\ldots,x_n)\in\widehat\Omega\cap B_\eps((x^k,\ldots,x^k)_n)$ and $(a_1,\ldots,a_n)\in\al\rho\B_{X^n}$ such that
condition \eqref{D2.5.3-1} is satisfied.
\end{definition}

Clearly, $\{\Omega_1,\ldots,\Omega_n\}$ is approximately stationary at $\{x^k\}\subset X$ if and only if it is approximately $\al$-sta\-tionary at $\{x^k\}$ for all $\al>0$.
An analogue of Corollary~\ref{P1.4}\,\eqref{P1.4.3} is true also for the sequential approximate $\al$-stationarity (for the definition of approximate $\al$-stationarity at a point we refer the reader to \cite{CuoKru25}).

As an application of the generalized separation Theorem~\ref{C3.9}, we prove dual necessary conditions for the sequential approximate $\al$-sta\-tio\-narity.

\begin{theorem}
[Sequential approximate $\al$-stationarity: generalized separation]
\label{T4.3}
Let $X$ be Banach, $\Omega_1,\ldots,\Omega_n$
be closed, and $\al>0$.
Suppose that $\{\Omega_1,\ldots,\Omega_n\}$ is approximately $\al$-sta\-tionary at $\{x^k\}\subset X$.
The following assertions hold true:
\begin{enumerate}
\item
\label{T4.2-1}
for any $\varepsilon>0$, $\be>\al$ and $\tau\in(0,1)$, there exist an integer $k>\eps\iv$, and points
$\hat a:=(a_1,\ldots,a_n)\in\eps\B_{X^n}$,
$\hat x:=(x_1,\ldots,x_n),\; \hat x':=(x'_1,\ldots,x'_n)\in\widehat\Omega\cap B_\eps((x^k,\ldots,x^k)_n)$, $x_0\in\eps\B_X$ and $\hat x^*:=(x_1^*,\ldots,x_n^*)\in N^C_{\widehat\Omega}(\hat x)$ such that
\begin{gather}
%\label{T4.3-3}
%\vertiii{(x_1-x^k,\ldots,x_n-x^k)}<\varepsilon,\;\;
%\vertiii{(x'_1-x^k,\ldots,x'_n-x^k)}<\varepsilon,\;\;
%\|x_0\|<\varepsilon,\\
\label{T4.3-5}
\Big\|{\sum_{i=1}^nx_i^*}\Big\|<\be,\quad
\vertiii{\hat x^*}=1,\\
\label{T4.3-2}
\ang{\hat x^*,(x_0,\ldots,x_0)_n+\hat a+\hat x'-\hat x}
%=\sum_{i=1}^{n}\langle x_i^*,x_0+a_i+x'_i-x_i\rangle
>\tau\vertiii{(x_0,\ldots,x_0)_n+\hat a+\hat x'-\hat x};
\end{gather}
\item
\label{T4.2-2}
if $X$ is Asplund,
then $N^C$ in \eqref{T4.2-1} can be replaced by $N^F$.
\end{enumerate}
\end{theorem}

\begin{proof}
Let $\eps>0$, $\be>\al$, {and} $\tau\in(0,1)$.
It is easy to check that the second inequality in \eqref{C} implies the next compatibility condition
for the dual norms
(with the same $\kappa_2>0$):
\begin{gather}
\label{C6}
\sum_{i=1}^n\|u_i^*\|\le \kappa_2\vertiii{(u_1^*,\ldots,u_n^*)}\quad
\text{for all}\;\; u_1^*,\ldots,u_n^*\in X^*.
\end{gather}
Choose a number $\xi>0$ so that
\begin{gather}
\label{T4.3P3}
2\xi<1-\tau,\quad
\xi^2<\eps,\quad\al\xi<\eps\AND \frac{\al+\kappa_2\xi}{1-\xi}<\be.
\end{gather}
By Definition~\ref{D3.5}, there exist
a $\rho\in(0,\xi^2)$, an integer $k>\xi^{-2}$, and points $\hat a:=(a_1,\ldots,a_n)\in\al\rho\B_{X^n}$ and
$\hat x':=(x_1',\ldots,x_n')\in\widehat\Omega\cap B_{\xi^2}((x^k,\ldots,x^k)_n)$ such that
$\bigcap_{i=1}^n(\Omega_i-x_{i}'-a_i)\cap (\rho\B)=\emptyset$.

Choose a $\rho'\in(\vertiii{\hat a}/\al,\rho)$,
and set $\eps':=\al\rho'$ and $\de:=\al\sqrt{\rho'}$.
Then {$\vertiii{\hat a}<\eps'$ and $\bigcap_{i=1}^n(\Omega_i-x_{i}'-a_i)\cap (\rho'\overline\B)=\emptyset$}.
By Theorem~\ref{C3.9}\,\eqref{C3.3-1},
%(see also Remark~\ref{R4}\,\eqref{R4-2}),
there {exist}
%points
$\hat x:=(x_1,\ldots,x_n)\in\widehat\Omega\cap B_\de(\hat x')$, {$x_0\in\rho'\B$} and $\hat x'^*:=(x_1'^*,\ldots,x_n'^*)\in(X^*)^n$ such that $\vertiii{\hat x'^*}=1$ and
\begin{gather}
\label{T4.3-4}
\de d\left(\hat x'^*,N^C_{\widehat\Omega}(\hat x)\right)+ \rho'\Big\|\sum_{i=1}^nx_i'^*\Big\|<\eps',\\
\label{T4.3-9}
\ang{\hat x'^*,(x_0,\ldots,x_0)_n+\hat a+\hat x'-\hat x}=
\vertiii{(x_0,\ldots,x_0)_n+\hat a+\hat x'-\hat x}.
\end{gather}
Thus, $k>\xi^{-2}>\varepsilon\iv$,
$\vertiii{\hat a}<\al\xi^2<\al\xi<\varepsilon$,
$\vertiii{(x_1-x^k,\ldots,x_n-x^k)}<\de< \al\xi<\varepsilon$ and
$\|x_0\|<\rho'<\xi^2<\varepsilon$.
By \eqref{T4.3-4},
\begin{gather*}
d\left(\hat x'^*,N^C_{\widehat\Omega}(\hat x)\right) <\frac{\eps'}{\de}= \frac{\al\rho'}{\al\sqrt{\rho'}}<\xi,\quad \Big\|\sum_{i=1}^nx_i'^*\Big\|<\frac{\eps'}{\rho'}=\al.
\end{gather*}
Thus, there is a $\hat z^*:=(z_1^*,\ldots,z_n^*)\in N^C_{\widehat\Omega}(\hat x)$ such that $\vertiii{\hat x'^*-\hat z^*}<\xi$, and consequently, ${0<1-\xi}<\vertiii{\hat z^*}<1+\xi$.
By \eqref{C6},
\begin{gather*}
\Big\|\sum_{i=1}^nz_i^*\Big\|\le \Big\|\sum_{i=1}^n x_i'^*\Big\|+
\sum_{i=1}^{n}\|z^*_i-x_i'^*\|<\al+\kappa_2\xi.
\end{gather*}
Set $\hat x^*:=(x_1^*,\ldots,x_n^*):=\dfrac{\hat z^*}{\vertiii{\hat z^*}}$.
Then, $\hat x^*\in N^C_{\widehat\Omega}(\hat x)$, $\vertiii{\hat x^*}=1$ and, in view of the last inequality in \eqref{T4.3P3},
\begin{gather*}
\Big\|\sum_{i=1}^nx_i^*\Big\| <\frac{\al+\kappa_2\xi}{1-\xi}<\be.
\end{gather*}
Hence, conditions \eqref{T4.3-5} are satisfied.
Moreover,
\begin{align*}
\vertiii{\hat x^*-\hat x'^*}
&\le \vertiii{\frac{\hat z^*}{{\vertiii{\hat z^*}}}-\hat z^*}+ \vertiii{\hat z^*-\hat x'^*}<\abs{{\vertiii{\hat z^*}}-1}+\xi<2\xi.
\end{align*}
Denote $m:=\vertiii{(x_0,\ldots,x_0)_n+\hat a+\hat x'-\hat x}$.
Condition \eqref{T4.3-2} follows from \eqref{T4.3P3}, \eqref{T4.3-9},  and the last estimate:
\begin{multline}
\label{T4.3-6}
\ang{\hat x^*,(x_0,\ldots,x_0)_n+\hat a+\hat x'-\hat x}>\ang{\hat x'^*,(x_0,\ldots,x_0)_n+\hat a+\hat x'-\hat x}- 2\xi m\\
=(1-2\xi)m>\tau m.
\end{multline}

Suppose $X$ is Asplund.
Let $\hat\tau\in(\tau+2\xi,1)$.
Application of Theorem~\ref{C3.9}\;\eqref{C3.3-2} with $\hat\tau$ in place of $\tau$ in the above proof justifies conditions \eqref{T4.3-5} with $\hat x^*\in N^F_{\widehat\Omega}(\hat x)$, while the factor $1-2\xi$ in \eqref{T4.3-6} needs to be replaced by $\hat\tau-2\xi$ leading to the same  estimate.
This again proves \eqref{T4.3-2}.
%The proof is complete.
\end{proof}

\begin{corollary}
[Sequential approximate stationarity: generalized separation]
\label{C3.8}
Let $X$ be Banach, and $\Omega_1,\ldots,\Omega_n$
be closed.
Suppose that $\{\Omega_1,\ldots,\Omega_n\}$ is approximately sta\-tionary at $\{x^k\}\subset X$.
The following assertions hold true:
\begin{enumerate}
\item
\label{C3.8.1}
for any $\varepsilon>0$ and $\tau\in(0,1)$, there exist {a $k>\eps\iv$, and points $x_0\in\eps\B_X$,
$\hat x:=(x_1,\ldots,x_n)$, $\hat x':=(x'_1,\ldots,x'_n)$ $\in\widehat\Omega\cap B_\eps((x^k,\ldots,x^k)_n)$, $\hat a:=(a_1,\ldots,a_n)\in\eps\B_{X^n}$, $\hat x^*:=(x_1^*,\ldots,x_n^*)\in N^C_{\widehat\Omega}(\hat x)$} such that
conditions \eqref{T1.4-2} and \eqref{T4.3-2} are satisfied;
\item
\label{C3.8.2}
for any $\varepsilon>0$, there exist {a $k>\eps\iv$, and points
$\hat x:=(x_1,\ldots,x_n)\in\widehat\Omega\cap B_\eps((x^k,\ldots,x^k)_n)$, $\hat x^*:=(x_1^*,\ldots,x_n^*)\in N^C_{\widehat\Omega}(\hat x)$} such that
conditions \eqref{T1.4-2} are satisfied;
\item
\label{C3.8.3}
if $X$ is Asplund,
then $N^C$ in \eqref{C3.8.1} and \eqref{C3.8.2} can be replaced by $N^F$.
\end{enumerate}
\end{corollary}

The necessary conditions in Corollary~\ref{C3.8} are applicable (with obvious amendments) to the stationarity and extremality properties in Definition~\ref{D3.1}.
In particular, we can formulate a result generalizing and extending the conventional extremal principle in Theorem~\ref{EP}.

\begin{corollary}
[Sequential extremal principle]
\label{C3.10}
Let $X$ be Banach, and $\Omega_1,\ldots,\Omega_n$
be closed.
Suppose that $\{\Omega_1,\ldots,\Omega_n\}$ is extremal at $\{x_i^k\}\subset\Omega_i$ $(i=1,\ldots,n)$.
The following assertions hold true:
\begin{enumerate}
\item
\label{C3.10.1}
for any $\varepsilon>0$ and $\tau\in(0,1)$, there exist {a $k>\eps\iv$, and points $x_0\in\eps\B_X$,
$\hat x:=(x_1,\ldots,x_n)$, $\hat x':=(x'_1,\ldots,x'_n)\in\widehat\Omega\cap B_\eps(x_1^k,\ldots,x_n^k)$, $\hat a:=(a_1,\ldots,a_n)\in\eps\B_{X^n}$, $\hat x^*:=(x_1^*,\ldots,x_n^*)\in N^C_{\widehat\Omega}(\hat x)$} such that
conditions \eqref{T1.4-2} and \eqref{T4.3-2} are satisfied;
\item
\label{C3.10.2}
for any $\varepsilon>0$, there exist {a $k>\eps\iv$, and points
$\hat x:=(x_1,\ldots,x_n)\in\widehat\Omega\cap B_\eps(x_1^k,\ldots,x_n^k)$, $\hat x^*:=(x_1^*,\ldots,x_n^*)\in N^C_{\widehat\Omega}(\hat x)$} such that
conditions \eqref{T1.4-2} are satisfied;
\item
\label{C3.10.3}
if $X$ is Asplund,
then $N^C$ in \eqref{C3.10.1} and \eqref{C3.10.2} can be replaced by $N^F$.
\end{enumerate}
\end{corollary}

\begin{proof}
The statement is a consequence of Corollary~\ref{C3.8} in view of Remark~\ref{R02}\,\eqref{R02.3} and the fact that
$\diam\{x_1^k,\ldots,x_n^k\}\to0$ as $k\to+\infty$
(see Definition~\ref{D3.1}\,\eqref{D3.1.1}).
\end{proof}

\begin{remark}
\label{R3}
\begin{enumerate}
\item
\label{R3.2}
The second assertions in Corollaries~\ref{C3.8} and \ref{C3.10} are simplified versions of the first ones.
They correspond to
dropping inequality \eqref{T4.3-2} together with the variables involved only in this condition.
A similar simplification can be made in assertion \eqref{T4.2-1} of Theorem~\ref{T4.3}.

\item
{If $x_i^k:=\bx$ for all $i=1,\ldots,n$ and $k\in\N$, and the maximum norm \eqref{norm} is used, then Corollary~\ref{C3.10}\,\eqref{C3.10.3} covers Theorem~\ref{EP}.}

\item
\label{R3.3}
Imposing certain sequential normal compactness assumptions (which are automatically satisfied in finite dimensions), one can formulate limiting versions of Theorem~\ref{T4.3} and Corollary~\ref{C3.8} in terms of certain types of limiting normal cones.
%see comments after Theorem~\ref{EP}.
This remark applies also to the statements in the rest of the paper.
\end{enumerate}
\end{remark}

%The Asplund space assertion \eqref{T4.2-2} in Theorem~\ref{T4.3} can be partially reversed; cf., e.g., \cite
%[Theorem~4.4]
%{CuoKru25}.

{The generalized separation conditions in Theorem~\ref{T4.3}, when formulated in terms of \Fr\ normal cones, are in fact sufficient for sequential approximate $\al$-stationarity (with a smaller $\al>0$)}; cf., e.g., \cite
%[Theorem~4.4]
{CuoKru25}.

\begin{theorem}
[Sequential generalized separation in Asplund spaces]
\label{T4.4}
Let $\{x^k\}\subset X$, $\al>0$ and $\be>0$.
Consider the following assertions:
\begin{enumerate}
\item\label{T4.4-1}
$\{\Omega_1,\ldots,\Omega_n\}$
is approximately $\al$-stationary at $\{x^k\}$;
\item
\label{T4.4-2}
for any $\varepsilon>0$ and $\tau\in(0,1)$, there exist {a $k>\eps\iv$, and points $x_0\in\eps\B_X$,
$\hat x:=(x_1,\ldots,x_n)$, $\hat x':=(x'_1,\ldots,x'_n)\in\widehat\Omega\cap B_\eps((x^k,\ldots,x^k)_n)$, $\hat a:=(a_1,\ldots,a_n)\in\eps\B_{X^n}$, $\hat x^*:=(x_1^*,\ldots,x_n^*)\in N^C_{\widehat\Omega}(\hat x)$} such that
conditions \eqref{T4.3-5} and \eqref{T4.3-2} are satisfied;
\item
\label{T4.4-3}
for any $\varepsilon>0$, there exist {a $k>\eps\iv$, and points
$\hat x:=(x_1,\ldots,x_n)\in\widehat\Omega\cap B_\eps((x^k,\ldots,x^k)_n)$,  $\hat x^*:=(x_1^*,\ldots,x_n^*)\in N^F_{\widehat\Omega}(\hat x)$} such that
conditions \eqref{T4.3-5} are satisfied.
\end{enumerate}
The following relations hold true:
{\renewcommand{\theenumi}{\rm\alph{enumi}}
\begin{enumerate}
\item\label{a1}
{\rm \eqref{T4.4-2} \folgt \eqref{T4.4-3}};
\item\label{b1}
if $X$ is Asplund, $\Omega_1,\ldots,\Omega_n$ are closed, and $\be>\al$, then
{\rm \eqref{T4.4-1}~\folgt\ \eqref{T4.4-2}};
\item\label{c1}
if $\al>\be$, then
{\rm \eqref{T4.4-3}~$\Rightarrow$~\eqref{T4.4-1}}.
\end{enumerate}
}
\end{theorem}

\begin{proof}
The implication in
\eqref{a1} is straightforward as \eqref{T4.4-3} is a simplified version of \eqref{T4.4-2}; see Remark~\ref{R3}\,\eqref{R3.2}.
The implication in \eqref{b1} is a direct consequence of Theorem~\ref{T4.3}\,\eqref{T4.2-2}.
We now prove the implication in \eqref{c1}.

Suppose that the second inequality in \eqref{C} is satisfied with some $\kappa_2>0$,
assertion \eqref{T4.4-3} holds true, and $\al>\be$.
Let $\eps>0$.
Then, there exist
an integer $k>\eps\iv$, and points
$\hat x:=(x_1,\ldots,x_n)\in\widehat\Omega\cap B_\eps((x^k,\ldots,x^k)_n)$ and $\hat x^*:=(x_1^*,\ldots,x_n^*)\in N^F_{\widehat\Omega}(\hat x)$ such that
conditions \eqref{T4.3-5} are satisfied.

Choose a $\xi\in(0,\al-\be)$.
Thus, $\xi':=\al-\be-\xi>0$.
By the definition of \Fr\ normal cone,
there is a $\rho\in(0,\varepsilon)$ such that
\begin{gather}
\label{T4.4P1}
\langle \hat x^*,\hat\omega-\hat x\rangle\le \dfrac{\xi}{\kappa_2+\al} \vertiii{\hat\omega-\hat x}<\xi\rho
\;\;
\text{for all}
\;\;
\hat\omega\in\widehat\Omega\;\;\text{with}\;\;
\vertiii{\hat\omega-\hat x}<(\kappa_2+\al)\rho.
\end{gather}
By the equality in \eqref{T4.3-5},
%we have $-\al\rho<\langle x^*,a\rangle<\al\rho$ for all $a\in X ^n$ with $\vertiii{a}<\al\rho$.
%Thus,
one can choose an $\hat a:=(a_1,\ldots,a_n)\in X^n$ such that
\begin{gather}
\label{T4.4P2}
\vertiii{\hat a}<\al\rho
\;\;
\text{and}
\;\;
\langle \hat x^*,\hat a\rangle
>(\al-\xi')\rho=(\be+\xi)\rho.
\end{gather}
We now show that $\bigcap_{i=1}^n(\Omega_i-x_i-a_i)\cap (\rho\B)=\emptyset$.
Indeed, suppose that $\omega_i-x_i-a_i=x_0$ for some $x_0\in\rho\B$ and $\hat\omega:= (\omega_1,\ldots,\omega_n)\in\widehat\Omega$, and all $i=1,\ldots,n$.
Then, in view of \eqref{C} and the first inequality in \eqref{T4.4P2},
\begin{align*}
\vertiii{\hat\omega-\hat x}
\le \vertiii{(x_0,\ldots,x_0)_n}+\vertiii{\hat a}\le \kappa_2\|x_0\|+\vertiii{\hat a}<(\kappa_2+\al)\rho,
\end{align*}
and, by \eqref{T4.4P1},
$\ang{\hat x^*,\hat\omega-\hat x}<\xi\rho$.
Combining this with the second inequality in \eqref{T4.4P2}, we obtain
\begin{align*}
\Big\langle \sum_{i=1}^{n}x^*_i,x_0\Big\rangle=\langle \hat x^*,(x_0,\ldots,x_0)_n \rangle=\langle \hat x^*,\hat\omega-\hat x\rangle-\langle \hat x^*,\hat a\rangle<-\be\rho.
\end{align*}	
On the other hand, by the inequality in \eqref{T4.3-5},
$\langle \sum_{i=1}^{n}x^*_i,x_0\rangle> -\be\rho,$
a contradiction.
Hence, condition \eqref{D2.5.3-1} is satisfied, and $\{\Omega_1,\ldots,\Omega_n\}$
is approximate $\al$-stationary at $\{x^k\}$.
%The proof is complete.
\end{proof}

The next corollary generalizes and improves the extended extremal principle in Theorem~\ref{T1.4}.

\begin{corollary}
[Sequential extended extremal principle]
\label{C4.5}
Let $X$ be Asplund, $\Omega_1,\ldots,\Omega_n$ %$(n>1)$
be closed, and $\{x^k\}\subset X$.
The following assertions are equivalent:
\begin{enumerate}
\item
\label{C4.5.1}
$\{\Omega_1,\ldots,\Omega_n\}$
is approximately stationary at
{$\{x^k\}$};
\item	
\label{C4.5.2}
for any $\varepsilon>0$ and $\tau\in(0,1)$, there exist {a $k>\eps\iv$, and points $x_0\in\eps\B_X$,
$\hat x:=(x_1,\ldots,x_n)$, $\hat x':=(x'_1,\ldots,x'_n)\in\widehat\Omega\cap B_\eps((x^k,\ldots,x^k)_n)$, $\hat a:=(a_1,\ldots,a_n)\in\eps\B_{X^n}$, $\hat x^*:=(x_1^*,\ldots,x_n^*)\in N^F_{\widehat\Omega}(\hat x)$} such that
conditions \eqref{T1.4-2} and \eqref{T4.3-2} are satisfied;
\item	
\label{C4.5.3}
for any $\varepsilon>0$, there exist {a $k>\eps\iv$, and points
$\hat x:=(x_1,\ldots,x_n)\in\widehat\Omega\cap B_\eps((x^k,\ldots,x^k)_n)$, $\hat x^*:=(x_1^*,\ldots,x_n^*)\in N^F_{\widehat\Omega}(\hat x)$}
such that
conditions \eqref{T1.4-2} are satisfied.
\end{enumerate}
\end{corollary}

\begin{remark}
\label{R3.8}
\begin{enumerate}
\item
\label{R3.8.1}
Implications \eqref{C4.5.2} \folgt\ \eqref{C4.5.3} \folgt\ \eqref{C4.5.1} in Corollary~\ref{C4.5} are true in the setting of an arbitrary normed vector space and not necessary closed sets $\Omega_1,\ldots,\Omega_n$.
The Asplund property of the space and closedness of the sets are only needed for implication \eqref{C4.5.1}~\folgt~\eqref{C4.5.2} which is a consequence of Theorem~\ref{T4.4}\,\eqref{b1}.
\item
{If $x^k:=\bx$ for all $k\in\N$, then Corollary~\ref{C4.5} covers Theorem~\ref{T1.4}}.
\end{enumerate}
\end{remark}

Reversing the conditions in Definition~\ref{D3.5}, we arrive at extensions of the \emph{transversality} properties discussed in \cite{Kru05,KruTha13,CuoKru21.2}.

\begin{definition}
[Sequential transversality]
\begin{enumerate}
\item
Let $\al>0$.
The collection $\{\Omega_1,\ldots,\Omega_n\}$ is $\al$-trans\-versal at $\{x^k\}\subset X$ if there is an $\varepsilon>0$ such that condition \eqref{P3.6P1} is satisfied
for all $\rho\in(0,\varepsilon)$, integers $k>\eps\iv$, and points
$(x_1,\ldots,x_n)\in\widehat\Omega\cap B_\eps((x^k,\ldots,x^k)_n)$ and $(a_1,\ldots,a_n)\in\al\rho\B_{X^n}$.
\item
The collection $\{\Omega_1,\ldots,\Omega_n\}$ is transversal at $\{x^k\}\subset X$ if it is $\al$-transversal at $\{x^k\}$ for some $\al>0$.
\end{enumerate}
\end{definition}

The statements of Theorem~\ref{T4.3} and its corollaries can also be easily ``reversed'' to produce a dual characterization of transversality.
For instance, Corollary~\ref{C4.5} leads to the following statement.

\begin{corollary}
[Sequential transversality: dual characterization]
\label{C3.12}
Let $X$ be Asplund, and $\Omega_1,\ldots,\Omega_n$
be closed.
The collection $\{\Omega_1,\ldots,\Omega_n\}$ is  transversal at $\{x^k\}\subset X$ if and only if there is an $\varepsilon>0$ such that
$\norm{\sum_{i=1}^{n}x^*_i}\ge\eps$
for all integers $k>\eps\iv$, and points
$\hat x:=(x_1,\ldots,x_n)\in\widehat\Omega\cap B_\eps((x^k,\ldots,x^k)_n)$ and $\hat x^*:=(x_1^*,\ldots,x_n^*)\in N^F_{\widehat\Omega}(\hat x)$ with
$\vertiii{\hat x^*}=1$.
\end{corollary}

\begin{remark}
\label{R3.11}
The ``only if'' part of Corollary~\ref{C3.12} is true in the setting of an arbitrary normed vector space and not necessary closed sets $\Omega_1,\ldots,\Omega_n$; cf. Remark~\ref{R3.8}\,\eqref{R3.8.1}.
\end{remark}

\if{
\paragraph*{Convex case.}
In the convex setting, the above mentioned properties admit simpler forms, and in some cases, are equivalent.
\if{
\NDC{19/10/25.
This subsection should be shortened.}
}\fi
\AK{24/01/25.
Yes, it should.
My suggestions:

1) All ``at a point'' assertions should be removed.

2) The argument in the proof of Lemma~\ref{L4.10} can be included in the other proof(s) where it is needed.

3) I have doubts about Proposition~\ref{P3.16} even if properly formulated and even in finite dimensions.I am not sure if the function mentioned in the comment is continuous.

4) $\al$-stationarity in Definition~\ref{D3.18} should be dropped.
I am not absolutely sure even if approximate $\al$-stationarity in Definition~\ref{D3.5} should be kept.
I included it to make a bridge to $\al$-transversality, but we do not use any of them in this paper.

5) Proposition~\ref{P3.19}\,\eqref{P3.19.2} can potentially be of interest.
I think $\eps$ in the last condition should be replaced by an arbitrarily large $\rho$.
On the other hand, it would probably be better to minimize discussions of $\al$-properties in this paper and drop this proposition.

6) What is approximate stationarity at $n$ sequences in Propositions~\ref{P3.16} and \ref{P3.5}?

7) Out of this subsection (paragraph?) I would probably keep only (an updated version of) Proposition~\ref{P3.5} once the issue with the approximate stationarity mentioned above is fixed.
}
}\fi

\section{Sequential minimality and stationarity}
\label{S4}
\if{
Optimization problem with set-valued constraints:
\begin{gather}
\label{P2}
\tag{$\mathcal{P}$}
\text{minimize }\;F_0(x)\quad \text{subject to }\; F_i(x)\cap K_i\ne\es\; (i=1,\ldots,n),\;\; x\in\Omega,
\end{gather}
where $F_i:X\toto Y_i$ $(i=0,\ldots,n)$ are mappings between normed spaces, $\Omega\subset X$, $K_i\subset Y_i$ $(i=1,\ldots,n)$,
{and $Y_0$ is equipped with a level-set mapping $L$.}
The ``functional'' constraints in \eqref{P2} can model a system of equalities and inequalities as well as more general operator-type constraints.
}\fi

To illustrate the model studied in the previous sections, we consider the following constrained optimization problem:
\begin{gather}
\label{P}
\tag{${P}$}
\text{minimize }\;f(x)\quad \text{subject to }\; x\in\Omega,
\end{gather}
where $\Omega$ is a nonempty subset of a normed vector space $X$.
Below, we recall the conventional definition of minimizing sequence and introduce its localized version.

\begin{definition}
[Minimizing sequence]
\label{D4.1}
\begin{enumerate}
\item
A sequence $\{x^k\}\subset\Omega$ is
minimizing for problem \eqref{P} if $f(x^k)\to\inf_\Omega f$.
\item
\label{D4.1.2}
A sequence $\{x^k\}\subset\Omega$ is
minimizing for problem \eqref{P} at level $\mu_0\in\R$ if $f(x^k)\to\mu_0$ as $k\to+\infty$,
and there exist a $\rho\in(0,+\infty]$ and a $k_0>0$ such that
\begin{gather}
\label{D4.1-1}
f(x)\ge\mu_0
\;\;
\text{for all}
\;\;
x\in\Omega\cap B_\rho(x^k)
\;\;
\text{and all integers}
\;\;
k>k_0.
\end{gather}
\end{enumerate}
\end{definition}

The assertions in the next proposition are immediate consequences of the definitions.
They show, in particular, that the properties in Definition~\ref{D4.1} are not too different.

\begin{proposition}
Let $\{x^k\}\subset\Omega$ and $\mu_0\in\R$.
The following assertions hold true:
\begin{enumerate}
\item
if $\{x^k\}$ is minimizing for problem \eqref{P} at level $\mu_0$ with some $\rho\in(0,+\infty]$, and $x^k=\bx$ for all $k\in\N$, then $f(\bx)=\mu_0=\min_{\Omega\cap B_\rho(\bx)}f$;
\item
if $\{x^k\}$ is minimizing for problem \eqref{P} and $\inf_\Omega f\in\R$, then it is minimizing for problem~\eqref{P} at level $\inf_\Omega f$ with $\rho=+\infty$;
\item
if $\{x^k\}$ is
minimizing for problem \eqref{P} at level $\mu_0$, then $\mu_0\ge\inf_\Omega f$;
\item
if $\{x^k\}$ is
minimizing for problem \eqref{P} at level $\mu_0$ with $\rho=+\infty$, then it is minimizing for problem \eqref{P} and $\mu_0=\inf_\Omega f$.
\end{enumerate}
\end{proposition}

\begin{corollary}
\label{C4.3}
Let $\inf_\Omega f\in\R$.
A sequence $\{x^k\}\subset\Omega$ is minimizing for problem \eqref{P} if and only if it is minimizing for problem~\eqref{P} at level $\inf_\Omega f$ with $\rho=+\infty$.
\end{corollary}

\begin{example}
[Minimizing sequence at level $0$]
\label{E4.2}
Let $\Omega=\R$, $f(x)=1/x$ for all $x\ne0$ and $f(0)=+\infty$; see Figure~\ref{fig5}.
Any sequence of real numbers $x_k\to+\infty$ is
minimizing for \eqref{P} at level $0$.
Observe that it is not a minimizing sequence.
\begin{figure}[H]
\centering
\includegraphics[width=0.495\linewidth]{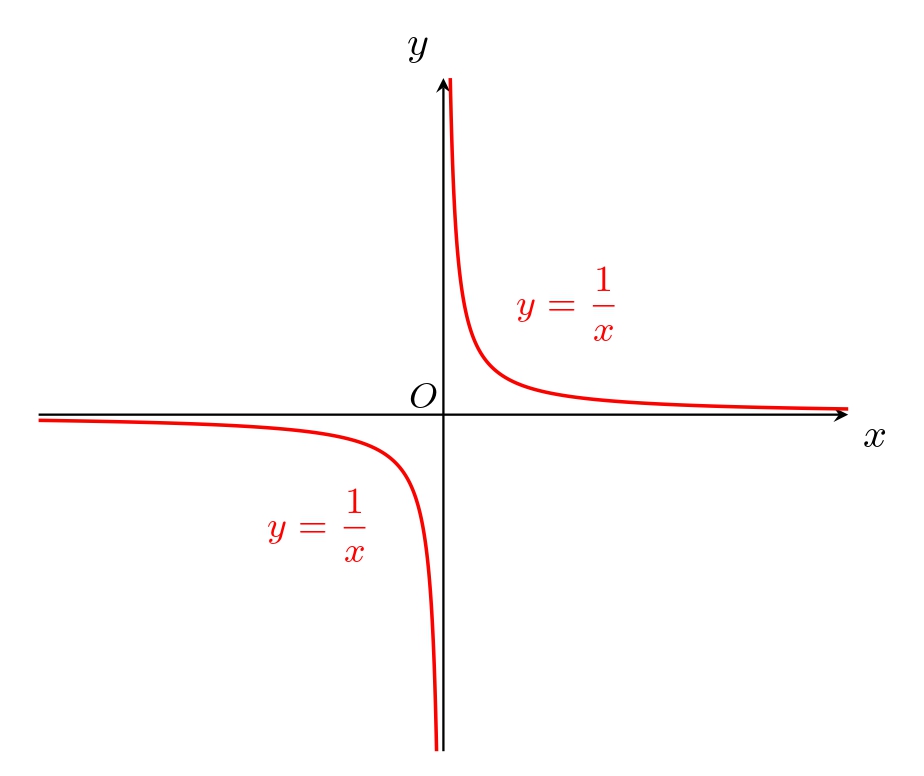}
\caption{Example~\ref{E4.2}}
\label{fig5}
\end{figure}
\end{example}
%\NDC{11/6/25.
%I think the picture should be moved to the previous page to prevent empty space.}

The stationarity properties in the next definition are counterparts of the corresponding ones in Definition~\ref{D3.1}.

\begin{definition}
[Stationarity]
%and approximate stationarity]
\label{D4.3}
\begin{enumerate}
\item
\label{D4.3.1}
A sequence $\{x^k\}\subset\Omega$ is firmly
$\inf$-stationary for problem \eqref{P} at level $\mu_0\in\R$ if $f(x^k)\to\mu_0$ as $k\to+\infty$, and there is a $\rho\in(0,+\infty]$ such that
%for any $\varepsilon>0$, there exist a $k\in\N$, $k>\eps\iv$ such that
\begin{gather}
\label{D4.3-1}
%\inf_{B_\rho(x^k)}f>\mu_0-\eps.\quad
\limsup_{k\to+\infty}\;\inf_{\Omega\cap B_\rho(x^k)}f=\mu_0.
\end{gather}
\item
\label{D4.3.2}
A sequence $\{x^k\}\subset\Omega$ is
$\inf$-stationary for problem \eqref{P} at level $\mu_0\in\R$ if $f(x^k)\to\mu_0$ as $k\to+\infty$,
and
\begin{gather}
\label{D4.1-2}
\limsup_{k\to+\infty,\,\rho\downarrow0} \frac{\inf_{\Omega\cap B_\rho(x^k)}f-f(x^k)}\rho=0.
\end{gather}
\item
\label{D4.3.3}
A sequence $\{x^k\}\subset X$ is approximately $\inf$-stationary for problem \eqref{P} at level $\mu_0\in\R$ if
%$f(x^k)\to\mu_0$ as $k\to+\infty$, and
%$d(x^k,\Omega)\to0$ as $k\to+\infty$, and	
\begin{gather}
\label{D4.1-3}
\limsup_{\substack{k\to+\infty,\,\rho\downarrow0\\ u^k\in\Omega,\,u^k-x^k\to0,\,f(u^k)\to\mu_0}} \frac{\inf_{\Omega\cap B_\rho(u^k)}f-f(u^k)}\rho=0.
\end{gather}
\end{enumerate}
\end{definition}
\if{
\AK{24/01/25.
Does the definition cover your model involving the ``\nbh\ of infinity'' and the model(s) in
\cite{KimTunSon25,kimsontuntuy23}?

%I am thinking about making amendments to Definitions~\ref{D3.1} and \ref{D4.1}.
What about dropping the second parts of Definitions~\ref{D1.3}, \ref{D3.1} and \ref{D4.1}?}

\NDC{25/01/25.
In the mentioned papers, only global minimality at `infinity' was considered, i.e. $\mu_0:=\inf_{\Omega}f$ where the minimum is not attained.
Definition~\ref{D4.1} covers this and includes the conventional minimality.

I think dropping the stationarity and approximate stationarity properties makes the paper neater.
From the reader's point of view, I think it is good since people are often scared of several definitions.
These properties/results can be studied in subsequent paper(s).
}
\AK{1/02/25.
I am only thinking about dropping stationarity.
Approximate stationarity seems important as a bridge to (quantitative) transversality.}

\NDC{3/2.25
At first glance, jumping from (i) to (iii) (skipping (ii)) in Definition~\ref{D4.3} does not seem natural to me.
}
}\fi

\begin{remark}
\label{R4}
\begin{enumerate}
\item\label{R4.0}
If $x^k=\bx$ for all $k\in\N$, the property in part \eqref{D4.3.1} of Definition~\ref{D4.3} coincides with that in Definition~\ref{D4.1}\,\eqref{D4.1.2}, and $f(\bx)=\mu_0=\min_{\Omega\cap B_\rho(\bx)}f$.
If, additionally, $\Omega=X$, the properties in parts \eqref{D4.3.2} and \eqref{D4.3.3} reduce to the, respectively, $\inf$-stationarity and approximate $\inf$-stationarity (weak $\inf$-stationarity) studied in \cite{Kru06.2,Kru09}.
\item
\label{R4.1}
The property in Definition~\ref{D4.1}\,\eqref{D4.1.2} implies firm stationarity in Definition~\ref{D4.3}\,\eqref{D4.3.1}, and
\eqref{D4.3.1}~$\Rightarrow$~\eqref{D4.3.2} $\Rightarrow$ \eqref{D4.3.3} in Definition~\ref{D4.3}.
The converse implications are not true in general. See \cite[Examples~1--4]{Kru06.2} for the case $x^k:=\bx$ $(k\in\N)$.
The sequential case is illustrated in Examples~\ref{E4.3}, \ref{E4.4} and \ref{E4.5} below.
\end{enumerate}
\begin{comment}
\if{\NDC{20/01/25.
I saw in your papers that you also defined (approximate) stationarity properties in terms of slopes.
I think it should be metioned as a remark.
}}\fi
\end{comment}
\end{remark}
\if{
\AK{1/02/25.
Firm inf-stationarity?}
}\fi
\begin{example}
[firm inf-stationarity]
\label{E4.3}
Let $\Omega=\R$, $f(x)=1/x$ for all $x\ne0$ and $f(0)=+\infty$ (see Example~\ref{E4.2}).
For the sequence $x_k:=-k$ $(k\in\N)$, we have $f(x^k)\to0$ as $k\to+\infty$, but the sequence
is not
minimizing
for \eqref{P} at level $0$ as $f(x)<0$ for all $x<0$.
For $\rho:=1$,
\begin{gather*}
 \limsup_{k\to+\infty}\;\inf_{\Omega\cap B_\rho(x^k)}f=\lim_{k\to+\infty}\frac{1}{-k+1}=0.
\end{gather*}
By Definition~\ref{D4.3} \eqref{D4.3.1}, $\{x^k\}$ is firmly
$\inf$-stationary for \eqref{P} at level $0$.
\end{example}

\begin{example}
[inf-stationarity]
\label{E4.4}
Let $\Omega=\R$, $f(x)=1/|x|-(x-k)^2$ for all $x\in[k-1/2,k+1/2)$,
$k=\pm1,\pm2,\ldots$ and $f(x)=7/4$ for all $x\in(-1/2,1/2)$; see Figure~\ref{fig4}.
For $x_k:=k$, we have
$f(x^k)=1/k\to0$ as $k\to+\infty$, and
for any $\rho>0$,
\begin{align*}
\inf_{B_\rho(x^k)}f= \frac1{k+\rho}-\min\Big\{\rho^2,\frac14\Big\}
\AND
\limsup_{k\to+\infty}\inf_{B_\rho(x^k)}f= -\min\Big\{\rho^2,\frac14\Big\}<0.
\end{align*}
By Definition~\ref{D4.3}\,\eqref{D4.3.1}, $\{x^k\}$ is not firmly
$\inf$-stationary for \eqref{P} at level 0.
At the same time,
\begin{align*}
\limsup_{\substack{k\to+\infty\\\rho\downarrow0}} \frac{\inf_{B_\rho(x^k)}f-f(x^k)}\rho=&
\lim_{\substack{k\to+\infty\\\rho\downarrow0}} \frac{\frac1{k+\rho}-\rho^2-\frac1{k}}\rho=
-\lim_{\substack{k\to+\infty\\\rho\downarrow0}} \Big(\frac1{k(k+\rho)}+\rho\Big)=0.
\end{align*}
By Definition~\ref{D4.3}\,\eqref{D4.3.2}, $\{x^k\}$ is $\inf$-stationary for \eqref{P} at level 0.
\begin{figure}[H]
\centering
\includegraphics[width=0.5\linewidth]{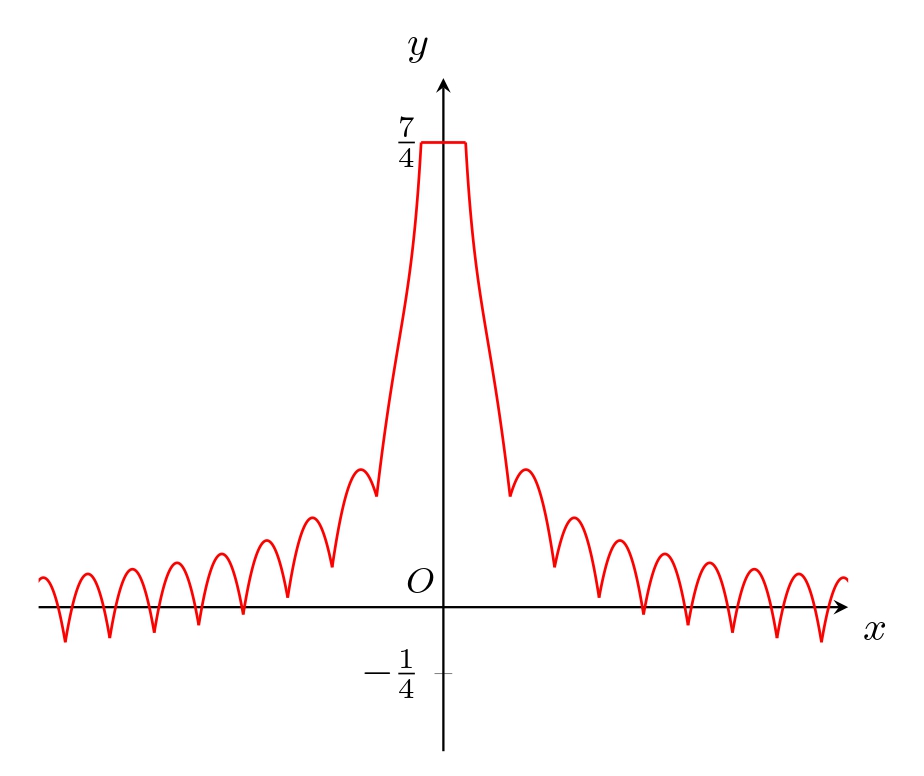}
\caption{Example~\ref{E4.4}}
\label{fig4}
\end{figure}
\if{
\AK{10/02/25.
Why is the plot stretched vertically so much?
Why two colours?
The example has been modified; please check.}
}\fi
\end{example}

\begin{example}
[approximate inf-stationarity]
\label{E4.5}
Let $\Omega=\R$.
Consider a continuous function
$$f(x)=
\begin{cases}
(x-2k\pi\iv)\sin\frac1{x-{2k}\pi\iv}& \text{if } 0<|x-{2k}\pi\iv|\le\pi\iv,\;\; k=0,\pm1,\pm2,\ldots,\\
0& \text{otherwise};
\end{cases}
$$
see Figure~\ref{fig6}.
For the sequence $x^k:={2k}\pi\iv$ $(k\in\N)$, we have $f(x^k)=0$ for all $k\in\N$, but the sequence is not
{inf-stationary}
for \eqref{P} at level $0$.
Indeed, for any $k\in\N$ and $\rho\in(0,\pi\iv)$, set $j:=\lceil\frac1{2\pi\rho}+\frac14\rceil$,
${\de}:=\frac2{(4j-1)\pi}$ and $x'^k:=x^k+\de$.
Then
$\frac1{2\pi+{\rho}\iv}<\de\le\rho$ and $f(x'^k)=\de\sin\frac1{\de}=\de\sin(-\frac\pi{2}+2j\pi)=-\de$,
and consequently,
\begin{gather*}
\limsup_{k\to+\infty,\,\rho\downarrow0} \frac{\inf_{B_\rho(x^k)}f-f(x^k)}\rho\le
\limsup_{k\to+\infty,\,\rho\downarrow0} \frac{f(x'^k)}\rho\le
%-\lim_{\rho\downarrow0} \frac1{\rho(2\pi+{\rho}\iv)}=
-\lim_{\rho\downarrow0} \frac1{2\pi\rho+1}=-1.
\end{gather*}
By Definition~\ref{D4.3} \eqref{D4.3.2}, $\{x^k\}$ is not inf-stationary for \eqref{P} at level $0$.

For each $k\in\N$, set $\de_k:=\frac1{2k\pi-\pi/2}$, $u^k:=x^k+\de_k$ and $\rho_k:=\frac1{4\pi k^2}$.
Then $u^k-x^k=\de_k\to0$, $f(u^k)=-\de_k\to0$ and $\rho_k\downarrow0$ as $k\to+\infty$.
Furthermore,
\begin{align*}
\frac1{\de_k+\rho_k} =\pi\Big(\frac1{2k-1/2}+\frac1{4k^2}\Big)\iv &>\pi\Big(\frac1{2k-1/2}+\frac1{(2k-3/2)(2k-1/2)}\Big)\iv \\&
=\pi(2k-3/2)=\de_k\iv-\pi,
\\
\frac1{\de_k-\rho_k} =\pi\Big(\frac1{2k-1/2}-\frac1{4k^2}\Big)\iv
&<\pi\Big(\frac1{2k-1/2}-\frac1{4k^2-1/4}\Big)\iv
\\&
=\pi(2k+1/2)=\de_k\iv+\pi.
\end{align*}
Thus, $\de_k\iv-\pi<(\de_k+\rho_k)\iv<\de_k\iv <(\de_k-\rho_k)\iv<\de_k\iv+\pi$.
In view of the definition of $f$, the point $u^k$ is the minimum of $f$ on $(u^k-\rho_k,u^k+\rho_k)$.
By Definition~\ref{D4.3}\,\eqref{D4.3.3},
$\{x^k\}$ is approximately inf-stationary  for  \eqref{P} at level $0$.
\begin{figure}[H]
\centering
\includegraphics[width=0.5\linewidth]{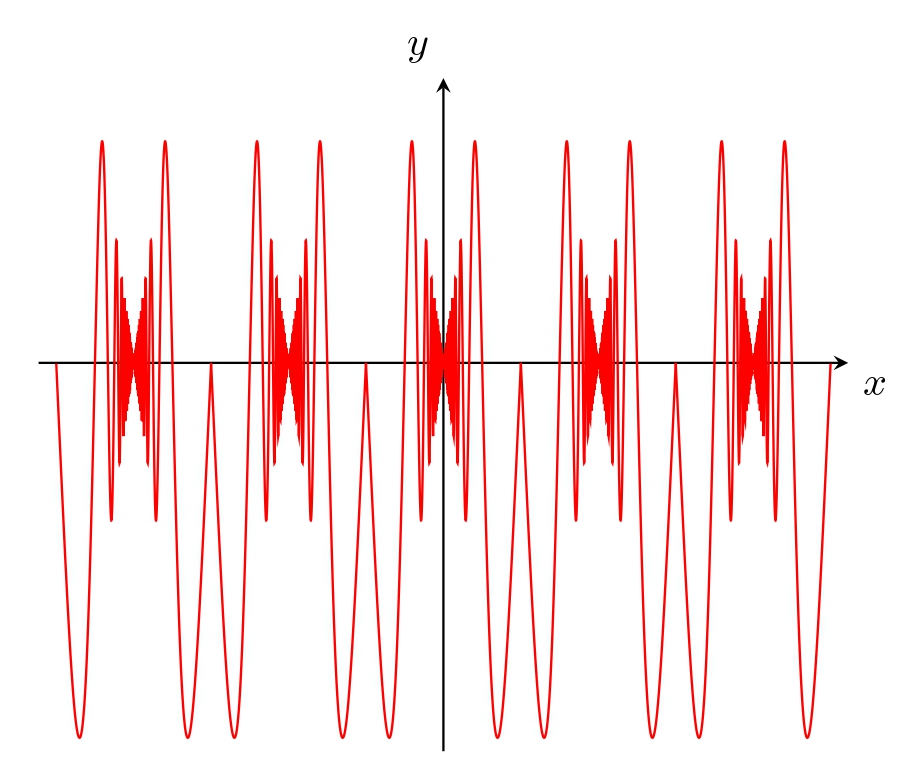}
\caption{Example~\ref{E4.5}}
\label{fig6}
\end{figure}
\end{example}
\if{
\AK{8/02/25.
I think the $y$ axis should be shortened at least twice.
Why two colours?

In this and the previous example, we just replicate the same pattern infinitely many times.
This does not agree well with our aim to characterize special behaviour ``at infinity''.}
\NDC{9/2/25. Yes, I agree. I think that these examples exist but it is not obvious to find.}
\AK{11/02/25.
I have updated the previous example.
It would be good to make the magnitude of the oscillations in the last one go to zero as $x\to+\infty$.}
\NDC{14/2/25. The current number of oscillations in the picture is the largest one that Overleaf allows me to run as a free account.}
}\fi

\section{Sequential necessary optimality and stationarity conditions}
\label{S5}

To embed problem \eqref{P} into the model studied in Section~\ref{S2},
given a $\mu_0\in\R$, we consider a pair of sets in $X\times\R$:
\begin{gather}	\label{Om}
\Omega_1:=\epi f
\;\;
\text{and}
\;\;
\Omega_2:=\Omega\times(-\infty,\mu_0].
\end{gather}
We are going to use the maximum product norms of the type \eqref{norm} and the corresponding dual norms.
\if{
\begin{gather*}
\|(u,\al)\|=\max\{\|u\|,|\al|\}\quad
\text{for all}\;\; u\in X,\;\;\al\in\R,
\\
\vertiii{(w_1,w_2)}=\max\{\|w_1\|,\|w_2\|\}\quad
\text{for all}\;\; w_1,w_2\in X\times\R.
\end{gather*}
}\fi

The next proposition relates the properties in Definition~\ref{D4.1} with the corresponding ones in
Definition~\ref{D3.1}.

\begin{proposition}
\label{P4.2}
Let $\mu_0\in\R$, and the sets $\Omega_1$ and $\Omega_2$ be given by \eqref{Om}.
\begin{enumerate}
\item
\label{P4.2.1}
If
%a sequence
$\{x^k\}\subset\Omega$ is
firmly $\inf$-stationary
%for problem
 for \eqref{P}
at level $\mu_0$, then
%the pair
$\{\Omega_1,\Omega_2\}$ is extremal at
%the sequences
$\{(x^k,f(x^k))\}\subset\Omega_1$ and $\{(x^k,\mu_0)\}\subset\Omega_2$.
\item\label{P4.2.2}
If
%a sequence
$\{x^k\}\subset\Omega$ is
$\inf$-stationary
for
%problem
\eqref{P} at level $\mu_0$, then
%the pair
$\{\Omega_1,\Omega_2\}$ is stationary at
%the sequences
$\{(x^k,f(x^k))\}\subset\Omega_1$ and $\{(x^k,\mu_0)\}\subset\Omega_2$.
\item
\label{P4.2.3}
If
$\{x^k\}\subset X$ is
approximately $\inf$-stationary
for
\eqref{P} at level $\mu_0$, then
$\{\Omega_1,\Omega_2\}$ is approximately stationary at
$\{(x^k,\mu_0)\}$.
\end{enumerate}
\end{proposition}

\begin{proof}
\begin{enumerate}
\item
Let a sequence $\{x^k\}\subset\Omega$ be firmly $\inf$-stationary
for problem \eqref{P} at level $\mu_0$, i.e., $f(x^k)\to\mu_0$ as $k\to+\infty$, and there is a $\rho\in(0,+\infty]$ such that condition \eqref{D4.3-1} is satisfied.
Then
\begin{gather}
\label{P4.2P1}
\lim_{k\to+\infty}\diam\{(x^k,f(x^k)),(x^k,\mu_0)\} =\lim_{k\to+\infty}|f(x^k)-\mu_0|=0,
\\
\label{P4.2P2}
\limsup_{k\to+\infty}\Big(\inf_{\Omega\cap B_\rho(x^k)}f-f(x^k)\Big)= \limsup_{k\to+\infty}\;\inf_{\Omega\cap B_\rho(x^k)}f-\mu_0=0.
\end{gather}
Let $\eps>0$.
By \eqref{P4.2P2}, there exists an integer $k>\eps\iv$ such that $\inf_{\Omega\cap B_\rho(x^k)}f-f(x^k)>-2\eps$.
Choose an $\al\in(0,\eps)$ so that
\begin{gather}
\label{P4.2P3}
\inf_{\Omega\cap B_\rho(x^k)}f-f(x^k)>-2\al.
\end{gather}
Set $a_1:=(0,-\al)$ and $a_2:=(0,\al)$.
Thus, $\vertiii{(a_1,a_2)}=\al<\eps$ and, thanks to \eqref{P4.2P3},
\begin{align*}
&(\Omega_1-(x^k,f(x^k))-a_1)\cap (\Omega_2-(x^k,\mu_0)-a_2)\cap (\rho\B_{X\times\R})
\\
=&\{(u-x^k,\mu)\mid u\in \Omega\cap B_\rho(x^k),\; \mu\ge f(u)-f(x^k)+\al,\; -\rho<\mu\le-\al\}=\es,
\end{align*}
i.e., $\{\Omega_1,\Omega_2\}$ is extremal at sequences $\{(x^k,f(x^k))\}\subset\Omega_1$ and $\{(x^k,\mu_0)\}\subset\Omega_2$.
\item
Let a sequence $\{x^k\}\subset\Omega$ be $\inf$-stationary
for problem \eqref{P} at level $\mu_0$, i.e., $f(x^k)\to\mu_0$ as $k\to+\infty$, and
condition \eqref{D4.1-2} is satisfied.
Then, condition \eqref{P4.2P1} holds true.
Let $\eps>0$.
By \eqref{D4.1-2}, there exist an integer $k>\eps\iv$, and a $\rho\in(0,\eps)$ such that $\inf_{\Omega\cap B_\rho(x^k)}f-f(x^k)> -2\eps\rho$.
Choose an $\al\in(0,\eps)$ so that
\begin{gather}
\label{P4.2P5}
\inf_{\Omega\cap B_\rho(x^k)}f-f(x^k)>-2\al\rho.
\end{gather}
Set $a_1:=(0,-\al\rho)$ and $a_2:=(0,\al\rho)$.
Thus, $\vertiii{(a_1,a_2)}=\al\rho<\eps\rho$ and, thanks to \eqref{P4.2P5},
\begin{align*}
&(\Omega_1-(x^k,f(x^k))-a_1)\cap (\Omega_2-(x^k,\mu_0)-a_2)\cap (\rho\B_{X\times\R})
\\
=&\{(u-x^k,\mu)\mid u\in \Omega\cap B_\rho(x^k),\; \mu\ge f(u)-f(x^k)+\al\rho,\; -\rho<\mu\le-\al\rho\}=\es,
\end{align*}
i.e., $\{\Omega_1,\Omega_2\}$ is stationary at sequences $\{(x^k,f(x^k))\}\subset\Omega_1$ and $\{(x^k,\mu_0)\}\subset\Omega_2$.
		
\item
Let a sequence $\{x^k\}\subset X$ be approximately $\inf$-stationary
for problem \eqref{P} at level $\mu_0$, i.e., condition \eqref{D4.1-3} is satisfied.
Let $\eps>0$.
By \eqref{D4.1-3}, there exist an integer $k>\eps\iv$, a $\rho\in(0,\eps)$ and $x\in\Omega$ such that $\|x-x^k\|<\eps$, $|f(x)-\mu_0|<\eps$, and $\inf_{\Omega\cap B_\rho(x)}f-f(x)>-2\eps\rho$.
Then $\vertiii{(x,f(x))-(x^k,\mu_0)}<\eps$ and $\vertiii{(x,\mu_0)-(x^k,\mu_0)}<\eps$.
Choose an $\al\in(0,\eps)$ so that
\begin{gather}
\label{P4.2P7}
\inf_{\Omega\cap B_\rho(x)}f-f(x)>-2\al\rho.
\end{gather}
Set $a_1:=(0,-\al\rho)$ and $a_2:=(0,\al\rho)$.
Thus, $\vertiii{(a_1,a_2)}=\al\rho<\eps\rho$ and, thanks to \eqref{P4.2P7},
\begin{align*}
&(\Omega_1-(x,f(x))-a_1)\cap (\Omega_2-(x,\mu_0)-a_2)\cap (\rho\B_{X\times\R})
\\
=&\{(u-x,\mu)\mid u\in \Omega\cap B_\rho(x),\; \mu\ge f(u)-f(x)+\al\rho,\; -\rho<\mu\le-\al\rho\}=\es,
\end{align*}
i.e., $\{\Omega_1,\Omega_2\}$ is approximately stationary at the sequence $\{(x^k,\mu_0)\}$.
\end{enumerate}
\end{proof}

In view of Remark~\ref{R4}\,\eqref{R4.1},
the next statement is a consequence of the sequential extremal principle in {Corollary~\ref{C3.10}} applied to the pair of sets $\{\Omega_1,\Omega_2\}$ given by \eqref{Om}.

\begin{theorem}
[Sequential necessary conditions]
\label{T4.6}
Let $X$ be Banach, $f:X\to\R_\infty$ be \lsc, and $\Omega\subset X$ be closed.
Suppose that $\{x^k\}\subset X$ is a minimizing sequence for problem \eqref{P} at level $\mu_0\in\R$.
The following assertions hold true:
\begin{enumerate}
\item
\label{T4.6.1}
for any $\varepsilon>0$, there exist an integer $k>\eps\iv$, and points
$(x_1,\mu_1)\in\epi f\cap B_\eps(x^k,\mu_0)$, $x_2\in\Omega\cap B_\eps(x^k)$, $(x_1^*,\nu_1)\in N^C_{\epi f}(x_1,\mu_1)$ and $x_2^*\in N^C_{\Omega}(x_2)$ such that
\begin{gather*}
\|x^*_1+x^*_2\|<\varepsilon
\;\;
\text{and}
\;\;
\|x^*_1\|+\|x^*_2\|+|\nu_1|=1.
\end{gather*}

\item
\label{T4.6.2}
If $X$ is Asplund, then $N^C$ in \eqref{T4.6.1} can be replaced by $N^F$.
\end{enumerate}
\end{theorem}

\begin{proof}
By Remark~\ref{R4}\,\eqref{R4.1}, $\{x^k\}$ is firmly $\inf$-stationary for problem \eqref{P} at level $\mu_0$.
By the assumptions, the sets $\Omega_1$ and $\Omega_2$ given by \eqref{Om} are closed.
By Proposition~\ref{P4.2}\,\eqref{P4.2.1}, $\{\Omega_1,\Omega_2\}$ is extremal at $\{(x^k,f(x^k))\}\subset\Omega_1$ and $\{(x^k,\mu_0)\}\subset\Omega_2$.
\if{
\NDC{18/2/25.
I think the correct statement should be  `extremal at $\{(x^k,f(x^k))\}\subset\Omega_1$ and $\{(x^k,\mu_0)\}\subset\Omega_2$'.
However, I think we should use Corollary \ref{C3.8} (for approximately $\inf$-stationary) instead.
In fact, if we use Corollary~\ref{C3.10}, additional estimates will be required, as we are working with the sequence
$(x^k,f(x^k))$ where $f(x^k)\to\mu_0$, but $f(x^k)$
does not appear in the conclusion.
}
}\fi

Let $\eps>0$.
Set $\eps':=\frac{\eps}{2+\eps}$ and observe that $\eps'\in(0,1)$ and $\eps'<\eps$.
By Corollary~\ref{C3.10}\,\eqref{C3.10.2} and taking into account that $f(x^k)\to\mu_0$, there exist an integer $k>(\eps')\iv$, and points
$(x_1,\mu_1)\in\epi f\cap B_{\eps'}(x^k,f(x^k))$, $(x_2,\mu_2)\in\Omega_2\cap B_{\eps'}(x^k,\mu_0)$, $(x_1^*,\nu_1)\in N^C_{\epi f}(x_1,\mu_1)$, and $(x_2^*,\nu_2)\in N^C_{\Omega_2}(x_2,\mu_2)$ such that $|f(x^k)-\mu_0|<\eps-\eps'$,
\begin{gather*}
\|x^*_1+x^*_2\|+|\nu_1+\nu_2|<\varepsilon'
\;\;
\text{and}
\;\;
\|x^*_1\|+\|x^*_2\|+|\nu_1|+|\nu_2|=1.
\end{gather*}
Then $(x_1,\mu_1)\in\epi f\cap B_\eps(x^k,\mu_0)$, $x_2\in\Omega\cap B_\eps(x^k)$, $x_2^*\in N^C_{\Omega}(x_2)$ and
\begin{gather*}
2(\|x^*_1\|+\|x^*_2\|+|\nu_1|)\ge \|x^*_1\|+\|x^*_2\|+2|\nu_1|=1+|\nu_1|-|\nu_2|\ge 1-|\nu_1+\nu_2|>1-\eps'>0.
\end{gather*}
Scaling the vectors
$(x_1^*,\nu_1)\in N^C_{\epi f}(x_1,\mu_1)$ and $x_2^*\in N^C_{\Omega}(x_2)$, one can ensure that (keeping the original notations) $\|x^*_1\|+\|x^*_2\|+|\nu_1|=1$ and $\|x^*_1+x^*_2\|<\frac{2\eps'}{1-\eps'}=\eps$.

If $X$ is Asplund, then $N^C$ in the above arguments can be replaced by $N^F$.
\end{proof}	

\begin{corollary}
\label{C4.7}
Under the assumptions of Theorem~\ref{T4.6}, one of the following assertions holds true:
\begin{enumerate}
\item
\label{C4.7.1}
there is an $M>0$ such that,
for any $\varepsilon>0$, there exist an integer $k>\eps\iv$, and points $x_1\in B_\eps(x^k)$, $x_2\in\Omega\cap B_\eps(x^k)$ such that
{$|f(x_1)-\mu_0|<\eps$},
and
\begin{gather}
\label{C4.7-1}
0\in\sd^{C}f(x_1)+N^C_{\Omega}(x_2)\cap (M\B_{X^*})	+\varepsilon\B_{X^*};
\end{gather}	
\item
\label{C4.7.2}
for any $\varepsilon>0$, there exist an integer $k>\eps\iv$, and points
$(x_1,\mu_1)\in\epi f\cap B_\eps(x^k,\mu_0)$, $x_2\in\Omega\cap B_\eps(x^k)$, $(x_1^*,\nu_1)\in N^C_{\epi f}(x_1,\mu_1)$ and $x_2^*\in N^C_{\Omega}(x_2)$ such that $-\eps<\nu_1\le0$
and
\begin{gather*}
\|x^*_1+x^*_2\|<\varepsilon,\;\;
\|x^*_1\|+\|x^*_2\|=1.	
\end{gather*}	
\end{enumerate}

If $X$ is Asplund, then $N^C$ and $\sd^{C}$ in the above assertions can be replaced by $N^F$ and $\sd^{F}$, respectively.
\end{corollary}

\begin{proof}
By Theorem~\ref{T4.6}, for any $j\in\N$, there exist an integer $k>j$, and points $(x_{1j},\mu_{1j})\in\epi f\cap B_{1/j}(x^k,\mu_0)$, $x_{2j}\in\Omega\cap B_{1/j}(x^k)$, $(x_{1j}^*,\nu_{1j})\in N^C_{\epi f}(x_{1j},\mu_{1j})$ and $x_{2j}^*\in N^C_\Omega(x_{2j})$
such that
$$\|x^*_{1j}+x^*_{2j}\|<1/j, \;\; \|x^*_{1j}\|+\|x^*_{2j}\|+|\nu_{1j}|=1.$$
Note that $\nu_{1j}\le0$ for all $j\in\N$.
We consider two cases.\\
\emph{Case 1.}
$\limsup_{j\to+\infty}|\nu_{1j}|>\al>0$.
Note that $\al<1$.
Set $M:=1/\al$.
Let $\eps>0$.
Choose a number $j\in\N$ so that
$j>(\al\eps)\iv$
and $|\nu_{1j}|>\al$.
Then $\mu_{1j}=f(x_{1j})$ and $x^*_{1j}/|\nu_{1j}|\in\sd^{C}f(x_{1j})$.
Note that $j>\varepsilon\iv$.
Set $x_1:=x_{1j}$, $x_2:=x_{2j}$,
$x_1^*:=x_{1j}^*/|\nu_{1j}|$, and
$x_2^*:=x_{2j}^*/|\nu_{1j}|$.
Then ${x_1\in B_\eps(x^k)}$, $x_2\in\Omega\cap B_\eps(x^k)$, $|f(x_1)-\mu_0|<\eps$,
$x_1^*\in\sd^{C}f(x_{1})$, $x_{2}^*\in N^C_\Omega(x_{2})$, $\|x_2^*\|<1/\al=M$,
and $\|x_1^*+x_2^*\|=\|x_{1j}^*+x_{2j}^*\|/|\nu_{1j}|<1/(\al j)<\eps$.
Hence, condition \eqref{C4.7-1} is satisfied.
\smallskip\\
\emph{Case 2.}
$\lim_{j\to+\infty}|\nu_{1j}|=0$.
Then $x_{1j}^*+x_{2j}^*\to0$ and $1\ge\|x_{1j}^*\|+\|x_{2j}^*\|\to1$ as $j\to+\infty$.
Let $\eps>0$.
Choose a number $j\in\N$ so that $j>\varepsilon\iv$,
\begin{gather*}
\ga:=\|x_{1j}^*\|+\|x_{2j}^*\|>0,
\;\;
|\nu_{1j}|/\ga<\varepsilon, \;\;
\|x_{1j}^*+x_{2j}^*\|/\ga<\varepsilon.
\end{gather*}
Set $x_1:=x_{1j}$, $\mu_1:=\mu_{1j}$, $x_2:=x_{2j}$, $x_1^*:=x_{1j}^*/\ga$, $x_2^*:=x_{2j}^*/\ga$, and $\nu_1:=\nu_{1j}/\ga$.
Then $(x_1,\mu_1)\in\epi f\cap B_\eps(x^k,\mu_0)$,
$x_2\in\Omega\cap B_\eps(x^k)$,
$(x_1^*,\nu_1)\in N^C_{\epi f}(x_1,\mu_1)$,
$x_2^*\in N^C_\Omega(x_{2})$, $-\eps<\nu_1\le0$,
${\|x_1^*\|+\|x_2^*\|}=1$,
and $\|x_1^*+x_2^*\|<\eps$.

If $X$ is Asplund, then
$N^C$ and $\sd^{C}$ in
the above arguments can be replaced by
$N^F$ and $\sd^{F}$, respectively.
\end{proof}

\begin{remark}
\begin{enumerate}
\item
The necessary conditions in Theorem~\ref{T4.6} and Corollary~\ref{C4.7} are applicable to any type of stationary sequences in Definition~\ref{D4.3}.
Moreover, the generalized separation Theorem~\ref{T4.3}, whose Corollary~\ref{C3.10} is the core tool in the proof of Theorem~\ref{T4.6}, allows one to derive necessary conditions for ``almost minimizing'' sequences.
\item
Part \eqref{C4.7.1} of Corollary~\ref{C4.7} gives a kind of multiplier rule (in the normal form), while part \eqref{C4.7.2} corresponds to `singular' behaviour of $f$ on $\Omega$ with the normal vector $(x_1^*,\nu_1)$ to the epigraph of $f$ being ``almost horizontal''.
If $\mu_1>f(x_1)$ in part \eqref{C4.7.2}, then $\nu_1=0$ and $x_1^*$ is normal to $\dom f$ at $x_1$.
\end{enumerate}
\end{remark}

The next qualification condition excludes the singular behavior in Corollary~\ref{C4.7}\,\eqref{C4.7.2}.

\begin{enumerate}
\item [ ]
\begin{enumerate}
\item [$(QC)_C$]
there is an $\varepsilon>0$ such that
$\|x_1^*+x_2^*\|\ge\eps$
for all integers $k>\eps\iv$, and points $(x_1,\mu_1)\in\epi f\cap B_\eps(x^k,\mu_0)$, $x_2\in\Omega\cap B_\eps(x^k)$, $(x_1^*,\nu_1)\in N^C_{\epi f}(x_1,\mu_1)$ and $x_2^*\in N^C_{\Omega}(x_2)$ such that $-\eps<\nu_1\le0$ and $\|x_1^*\|+\|x_2^*\|=1$.
\end{enumerate}
\end{enumerate}

We denote by $(QC)_F$ the analogue of $(QC)_C$ with $N^F$ and $\sd^{F}$ in place of $N^C$ and $\sd^{C}$, respectively.
Clearly, $(QC)_C$\ \folgt\ $(QC)_F$.
\if{
Thanks to the representation of the Clarke normal cone in \cite[Theorem~5.2.18]{BorZhu05},
if $\dim X<+\infty$, the conditions are equivalent.

\AK{16/02/25.
Can anything be said in infinite dimensions?}
\NDC{17/2/25.
I think that not much can be said in infinite dimensions beyond convexity or regularity assumptions, under which the Clarke normal cone and the Fréchet normal cone coincide at the tail of sequences.
}
{
\begin{lemma}
Let $X$ be a vector space,
vectors $x$ and $y$ are convex combination of vectors $x_1,\ldots,x_n$ and $y_1,\ldots,y_m$, respectively, and $ \varepsilon>0$.
Then $x+y$ is an element of the set ${\rm{conv}}\{x_i+y_j\mid i=1,\ldots,n\;\;\text{and}\;\;j=1,\ldots,m\}$.
\end{lemma}	
\begin{proof}
Under the assumption made, there exist $\la_1\ldots,\la_n\ge 0$ and $\mu_1,\ldots,\mu_m\ge 0$ such that $x=\sum_{i=1}^{n}\la_ix_i$, $y=\sum_{j=1}^{m}\mu_j y_j$, and $\sum_{i=1}^{n}\la_i=\sum_{j=1}^{m}\mu_j=1$.
Then
\begin{align*}
x+y
&=\left(\sum_{i=1}^{n}\la_ix_i\right)\cdot\left(\sum_{j=1}^{m}\mu_j\right)+\left(\sum_{i=1}^{n}\la_i\right)\cdot\left(\sum_{j=1}^{m}\mu_j y_j\right)\\
&=\sum_{i=1}^{n}\sum_{j=1}^{m}\la_i\mu_j x_i+\sum_{i=1}^{n}\sum_{j=1}^{m}\la_i\mu_i x_iy_j\\
&=\sum_{i=1}^{n}\sum_{j=1}^{m}\la_i\mu_j(x_i+v_j).
\end{align*}	
Note that $\sum_{i=1}^{n}\sum_{j=1}^{m}\la_i\mu_j=1$.
The proof is complete.
\end{proof}	
}
{
The next statement is a consequence of \cite[Theorem~5.2.18]{BorZhu05}.
\begin{lemma}
Let $\dim X<+\infty$, $\Omega\subset X$, and $\bx\in\Omega$.
Then
\begin{align*}
N^C_{\Omega}(\bx)={\rm{\overline{conv}}}\{x^*\in X^*\mid \exists \{x^j\}\subset\Omega,&\;x^{*j}\in N^F_{\Omega}(x^j)\;\;\text{for all}\;\; j\in\N,\\
&\text{such that}\;\;
x^j\to\bx\;\text{and}\; x^{*j}\rightarrow x^*\;\text{as}\;j\to+\infty\}.
\end{align*}	
\end{lemma}
}
{
\begin{proposition}
Suppose that $\dim X<+\infty$.
Then, $(QC)_F$ \folgt $(QC)_C$.
\end{proposition}	
\begin{proof}
Suppose that  $\dim X<+\infty$ and condition $(QC)_F$ is satisfed.
Then there is an $\varepsilon>0$ such that
$\|x_1^*+x_2^*\|\ge\eps$
for all integers $k>\eps\iv$, and points $(x_1,\mu_1)\in\epi f\cap B_\eps(x^k,\mu_0)$, $x_2\in\Omega\cap B_\eps(x^k)$, $(x_1^*,\nu_1)\in N^F_{\epi f}(x_1,\mu_1)$ and $x_2^*\in N^F_{\Omega}(x_2)$ such that $-\eps<\nu_1\le0$ and $\|x_1^*\|+\|x_2^*\|=1$.\\
Choose an $\varepsilon'\in(0,\varepsilon)$.
Take arbitrary integers $k>{\eps'}\iv$, $(x_1,\mu_1)\in\epi f\cap B_{\eps'}(x^k,\mu_0)$, $x_2\in\Omega\cap B_{\eps'}(x^k)$, $(x_1^*,\nu_1)\in N^C_{\epi f}(x_1,\mu_1)$ and $x_2^*\in N^C_{\Omega}(x_2)$ such that $-\eps'<\nu_1\le0$ and $\|x_1^*\|+\|x_2^*\|=1$.
\end{proof}
}
\NDC{18/2/24.
I have failed to prove $(QC)_F$ \folgt $(QC)_C$ in finite dimensions.
By Lemma~5.6, $x^*_1$ and $x^*_2$ (Clarke dual vectors) are limits of sequences $\{x_1^{*k}\}$ and $\{x_2^{*k}\}$, where each sequence is a convex combination of Fr\'echet dual vectors $x_1^{*k}(i_k)$ and $x_2^{*k}(j_k)$ with a property that
 $\|x_1^{*k}(i_k)+x_2^{*k}(j_k)\|>\varepsilon$ for any $i_k$ and $j_k$.
By Lemma 5.5, the sum sequence $x_1^{*k}+x_2^{*k}$ is a convex combination of vectors $x_1^{*k}(i_k)+x_2^{*k}(j_k)$.
In general, we do not have the lower estimate for $x_1^{*k}+x_2^{*k}$.
For example $(0,0)$ is a convex combination of $(-1,0)$ and $(1,0)$, but $\|(-1,0)\|=\|(1,0)\|=1$ and $\|(0,0)\|=0$.
}
}\fi
\if{
\NDC{26/01/25. Should the above claim be formulated explicitly? Or we mentioned  it in subsequent paper(s).}
\AK{28/01/25.
Corollary~\ref{C4.30}.}
\NDC{26/01/25. Should the convetional nonlinear programming problems with equality and inequality constraints be formulated.}
\AK{28/01/25.
Future work, item (i).}
}\fi

The next statement is a direct consequence of Corollary~\ref{C4.7}.

\begin{corollary}
\label{C4.30}
Suppose the assumptions of Theorem~\ref{T4.6} and condition $(QC)_C$ are satisfied.
Then assertion \eqref{C4.7.1} in Corollary~\ref{C4.7} holds true.
	
If $X$ is Asplund and condition $(QC)_F$ is satisfied,
then assertion \eqref{C4.7.1} in Corollary~\ref{C4.7} holds true with $N^F$ and $\sd^{F}$ in place of $N^C$ and $\sd^{C}$, respectively.
\end{corollary}

Condition $(QC)_F$ is ensured by the transversality of the sets $\epi f$ and $\Omega\times\R$.

\begin{proposition}
\label{P4.10}
Let $\{x^k\}\subset\Omega$ and $\mu_0\in\R$.
If $\{\epi f,\Omega\times\R\}$ is transversal at $\{(x^k,\mu_0)\}$,
then condition $(QC)_F$ holds true.
\end{proposition}

\begin{proof}
Let $\{\epi f,\Omega\times\R\}$ be transversal at $\{(x^k,\mu_0)\}$.
If $(x^*,\nu)\in N^F_{\epi f}(x,\mu)$ for some $(x,\mu)\in\epi f$, then $\nu\le0$.
If $(x^*,\nu)\in N^F_{\Omega\times\R}(x,\mu)$ for some $(x,\mu)\in\Omega\times\R$, then $\nu=0$.
By Corollary~\ref{C3.12} and Remark~\ref{R3.11},
there is an $\varepsilon>0$ such that
$\max\{\|x_1^*+x_2^*\|,|\nu_1|\}\ge\eps$
for all integers $k>\eps\iv$, and points
$(x_1,\mu_1)\in\epi f\cap B_\eps(x^k,\mu_0)$, $x_2\in\Omega\cap B_\eps(x^k)$, $(x_1^*,\nu_1)\in N^F_{\epi f}(x_1,\mu_1)$ and $x_2^*\in N^F_{\Omega}(x_2)$ with
$\|x_1^*\|+\|x_2^*\|+|\nu_1|=1$.
%Choose an $\eps'\in(0,\eps)$ such that

Take any integer $k>\eps\iv$, and points
$(x_1,\mu_1)\in\epi f\cap B_\eps(x^k,\mu_0)$, $x_2\in\Omega\cap B_\eps(x^k)$, $(x_1^*,\nu_1)\in N^F_{\epi f}(x_1,\mu_1)$ and $x_2^*\in N^F_{\Omega}(x_2)$ such that $-\eps<\nu_1\le0$ and $\|x_1^*\|+\|x_2^*\|=1$.
Set $\al:=1+|\nu_1|$, $x_1'^*:=x_1^*/\al$, $x_2'^*:=x_2^*/\al$, $\nu_1':=\nu_1/\al$.
Then $(x_1'^*,\nu_1')\in N^F_{\epi f}(x_1,\mu_1)$, $x_2'^*\in N^F_{\Omega}(x_2)$ and
$\|x_1'^*\|+\|x_2'^*\|+|\nu_1'|=1/\al+|\nu_1|/\al=1$.
Hence, $\max\{\|x_1'^*+x_2'^*\|,|\nu_1'|\}\ge\eps$, and consequently, $\max\{\|x_1^*+x_2^*\|,|\nu_1|\}\ge\eps\al\ge\eps$.
Since $|\nu_1|<\eps$, the last inequality yields $\|x_1^*+x_2^*\|\ge\eps$.
Thus, condition $(QC)_F$ holds true.
\end{proof}

The transversality condition in Proposition~\ref{P4.10} is satisfied, for instance, if $f$ is Lipschitz continuous (near a tail of the sequence $\{x^k\}$)
or when (a tail of) the sequence $\{x^k\}$ lies in $\Int\Omega$.
\if{
\NDC{18/2/25.
I think it should be mentioned somewhere that both $(QC)_C$ and $(QC)_F$ hold true when $f$ is Lipschitz continuous (\cite{CuoKru25}).}
}\fi
It is not difficult to show that these conditions ensure both $(QC)_F$ and $(QC)_C$; cf. \cite{CuoKru25}.

The next example illustrates Corollary~\ref{C4.30}.
{It shows, in particular, that, similar to the classical analysis, sequential necessary conditions correspond to stationary sequences, which are not in general minimizing.}

\begin{example}
[Sequential necessary conditions]
{Let $\Omega=\R$, $f(x)=1/x$ for all $x\ne0$, ${f(0)=+\infty}$, and $\mu_0=0$.
The function $f$ is Lipschitz continuous on $(-\infty,1]\cup[1,+\infty)$.
We obviously have $N_\Omega^C(x)=N_\Omega^F(x)=\{0\}$ for all $x\in\Omega$,  $\sd^Cf(x)=\sd^Ff(x)=\{-1/x^2\}$ for all $x\ne0$, and $\sd^Cf(0)=\sd^Ff(0)=\es$.
Hence, given any $\eps>0$, conditions $|f(x)|<\eps$ and $0\in\sd f(x)+(-\eps,\eps)$ are trivially satisfied when $|x|>\max\{\eps\iv,1\}$.
Hence, assertion \eqref{C4.7.1} in Corollary~\ref{C4.7} holds true for any sequence $\{x_k\}\in\R$ with $|x_k|\to+\infty$.
The natural candidates are $x_k=1/k$ and $x_k=-1/k$ $(k\in\N)$.
(One could also consider mixtures of the two sequences, though it does not make much sense.)
Thus, conditions $(QC)_C$ and $(QC)_F$ are satisfied for both sequences.
They both are firmly inf-stationary at level $0$ in the sense of Definition~\ref{D4.3}, but only the first one is minimizing at level $0$ in the sense of Definition~\ref{D4.1}\,\eqref{D4.1.2}.}
\end{example}	

\begin{remark}
\begin{enumerate}
\item
The model studied in Section~\ref{S2} allows one to derive necessary optimality and stationarity conditions in more general than \eqref{P} optimization problems with functional and geometric constraints, and vector or set-valued objectives.
\item
{Similar to the classical necessary optimality conditions transforming the task of solving
an optimization problem to that of a system of equations,
the sequential necessary conditions in Theorem~\ref{T4.6} and its corollaries transform the task of finding a minimizing sequence to that of solving a generalized equation of the type \eqref{C4.7-1}.
This is not an easy job in general as with all sorts of ``fuzzy'' conditions in optimization.
Nevertheless, such conditions can still be helpful for testing particular sequences.
For other necessary and sufficient conditions for minimizing sequences, the reader is referred to \cite[Section~6]{HuaNgPen00} and the references therein.}
\end{enumerate}
\end{remark}

\section{Conclusions}
\label{Conclusions}

The sequential extremality (together with sequential versions of the related concepts of stationarity, approximate stationarity and transversality) of a finite collection of sets are studied.
The properties correspond to replacing a fixed point (extremal point) in the intersection of the sets by a collection of sequences of points in the individual sets with the distances between the corresponding points tending to
zero. This allows one to consider collections of unbounded sets with empty intersection.

The sequential extremal principle extending the conventional one is established in terms of Fréchet and Clarke normal cones.
This result can replace the conventional extremal principle when proving optimality, stationarity, transversality and regularity conditions, and calculus formulas in more general settings involving unbounded sets.
In this paper, as an illustration, it is used to derive sequential necessary conditions for minimizing (and more general firmly stationary, stationary and approximately stationary) sequences in a scalar optimization problem with a geometric constraint.

Other potential applications worth being studied:
\begin{itemize}
\item
sequential necessary optimality and stationarity conditions for optimization problems with scalar, vector and set-valued objectives and several functional and geometric constraints;
\item
sequential transversality and subtransversality properties of collections of sets;
\item
sequential metric regularity and subregularity of set-valued mappings;
\item
sequential error bounds of extended-real-valued functions;
\item
sequential qualification conditions;
\item
sequential extensions of limiting normal cones, subdifferentials and coderivatives, and their calculus.
\end{itemize}

\section*{Acknowledgments}

The authors would like to thank Professor Tien-Son Pham for attracting their attention to optimality concepts on unbounded sets and fruitful discussions.
%A part of the work was done during Alexander Kruger's stay at the Vietnam Institute for Advanced Study in Mathematics in Hanoi.
%He is grateful to the Institute for its hospitality and supportive environment.

{We thank the reviewers for their comments and suggestions which helped us improve the manuscript.}

\section*{Disclosure statement}
No potential conflict of interest was reported by the authors.

\section*{Funding}
%Nguyen Duy Cuong has been supported by the Postdoctoral Scholarship Programme of Vingroup Innovation Foundation (VinIF) code VINIF.2023.STS.54.
Nguyen Duy Cuong is supported by Vietnam National Program for the Development of Mathematics 2021-2030 under grant number B2023-CTT-09.

\section*{ORCID}
Nguyen Duy Cuong http://orcid.org/0000-0003-2579-3601
\\
Alexander Y. Kruger http://orcid.org/0000-0002-7861-7380
\def\cprime{$'$} \def\cftil#1{\ifmmode\setbox7\hbox{$\accent"5E#1$}\else
  \setbox7\hbox{\accent"5E#1}\penalty 10000\relax\fi\raise 1\ht7
  \hbox{\lower1.15ex\hbox to 1\wd7{\hss\accent"7E\hss}}\penalty 10000
  \hskip-1\wd7\penalty 10000\box7} \def\cprime{$'$} \def\cprime{$'$}
  \def\cprime{$'$} \def\cprime{$'$} \def\cprime{$'$}
  \def\Dbar{\leavevmode\lower.6ex\hbox to 0pt{\hskip-.23ex \accent"16\hss}D}
  \def\cfac#1{\ifmmode\setbox7\hbox{$\accent"5E#1$}\else
  \setbox7\hbox{\accent"5E#1}\penalty 10000\relax\fi\raise 1\ht7
  \hbox{\lower1.15ex\hbox to 1\wd7{\hss\accent"13\hss}}\penalty 10000
  \hskip-1\wd7\penalty 10000\box7} \def\cprime{$'$}

%\addcontentsline{toc}{section}{References}
%\bibliography{BUCH-kr,Kruger,KR-tmp,tmp}
%\bibliographystyle{tfnlm}
\end{document}